\documentclass[EJP]{ejpecp}
\usepackage{bbm}

\SHORTTITLE{An almost sure CLT for stretched polymers}

\TITLE{An almost sure CLT for stretched polymers}

\AUTHORS{Dmitry Ioffe\footnote{Technion 
    \EMAIL{ieioffe@ie.technion.ac.il}} \and 
Yvan Velenik\footnote{Universit\'e de Gen\`eve
\EMAIL{Yvan.Velenik@unige.ch}}
}


\AMSSUBJ{60F05, 60K35, 82B41, 82B44, 82D60} 

\SUBMITTED{August 9, 2012} 
\ACCEPTED{November 2, 2013} 

\VOLUME{0}
\YEAR{0}
\PAPERNUM{0}
\DOI{vVOL-PID}

\ABSTRACT{
We prove an almost sure {central limit theorem} (CLT) for spatial 
extension of stretched 
{(meaning subject to a non-zero pulling force)} 
polymers at very weak disorder in all dimensions $d+1\geq
4$.
}

\newtheorem{bigthm}{Theorem}   
\renewcommand{\thebigthm}{\Alph{bigthm}}  

\newcommand{\sumtwo}[2]{\sum_{\substack{#1 \\ #2}}} 
\newcommand{\abs}[1]{\left| #1\right|}

\newcommand{\calA}{\mathcal{A}}

\newcommand{\calF}{\mathcal{F}}

\newcommand{\calH}{\mathcal{H}}
\newcommand{\calI}{\mathcal{I}}

\newcommand{\calS}{\mathcal{S}}
\newcommand{\calT}{\mathcal{T}}
\newcommand{\calU}{\mathcal{U}}

\newcommand{\calY}{\mathcal{Y}}

\newcommand{\bbC}{\mathbb{C}}

\newcommand{\bbE}{\mathbb{E}}

\newcommand{\bbL}{\mathbb{L}}

\newcommand{\bbN}{\mathbb{N}}

\newcommand{\bbP}{\mathbb{P}}
\newcommand{\bbQ}{\mathbb{Q}}
\newcommand{\bbR}{\mathbb{R}}

\newcommand{\bbV}{\mathbb{V}}

\newcommand{\bbZ}{\mathbb{Z}}

\newcommand{\sfe}{{\sf e}}
\newcommand{\sfp}{{\sf p}}

\renewcommand{\emptyset}{\varnothing}

\newcommand{\Zd}{\bbZ^d}
\newcommand{\Zdo}{\bbZ^{d+1}}

\newcommand{\setof}[2]{\left\{#1 \,:\, #2 \right\}}
\newcommand{\given}{\,|\,}

\newcommand{\lb}{\left(}
\newcommand{\rb}{\right)}
\newcommand{\lbr}{\left\{}
\newcommand{\rbr}{\right\}}

\newcommand{\dd}{{\rm d}}

\newcommand{\case}[1]{C{\small ASE}\,#1.}

\newcommand{\Var}{\bbV{\rm ar}}

\newcommand{\tq}[1]{t^\omega_{#1}}
\newcommand{\fq}[1]{f^\omega_{#1}}
\newcommand{\fxq}[1]{f^{\theta_x \omega}_{#1}}
\newcommand{\fxpq}[1]{f^{\theta_{x^\prime} \omega}_{#1}}
\newcommand{\ta}[1]{{\mathbf t}_{#1}}
\newcommand{\fa}[1]{{\mathbf f}_{#1}}

\newcommand{\RWP}{\bbP_{\scriptscriptstyle\rm RW}}
\newcommand{\RWPC}{\bbP_{\scriptscriptstyle\rm RW}^\otimes}
\newcommand{\RWEC}{\bbE_{\scriptscriptstyle\rm RW}^\otimes}

\newcommand{\RWX}{\mathbf{X}}
\newcommand{\RWL}{\mathbf{L}}

\newcommand{\qS}[2]{\calS_{#1}^{#2}} 
\newcommand{\aS}[1]{{\mathbf S}_{#1}} 

\newcommand{\cHm}[1]{\calH^-_{#1}}

\newcommand{\One}{\mathbbm{1}}

\newcommand{\smo}[1]{{\mathrm o}\lb #1\rb }

\newcommand{\df}{\stackrel{\Delta}{=}}

\newcommand{\eqvs}{\stackrel{\sim}{=}}
\newcommand{\leqs}{\lesssim}            
\newcommand{\geqs}{\gtrsim}             

\newcommand{\xpr}{x^\prime}
\newcommand{\ypr}{y^\prime}

\newcommand{\mpr}{m^\prime}
\newcommand{\Xpr}{X^\prime}
\newcommand{\Lpr}{L^\prime}

\newcommand{\kpr}{k^\prime}
\newcommand{\RWXpr}{\RWX^\prime}
\newcommand{\RWLpr}{\RWL^\prime}

\newcommand{\be}{\begin{equation}}
\newcommand{\ee}{\end{equation}}

\begin{document}


\section{Introduction and Results}

Directed polymers in random media were introduced in~\cite{HH85} as an effective
model of Ising interfaces in systems with random impurities. The precise
mathematical formulation appeared in the seminal paper~\cite{ImSp88}, which
triggered a wave of subsequent investigations. The model of directed polymers
can be described as follows. Let $\eta=(\eta_k)_{0\leq k\leq n}$ be a
nearest-neighbour path on $\Zd$ starting at $0$, and let
$\gamma=(\gamma_k)_{0\leq k\leq n}$ with $\gamma_k=(k,\eta_k)$ be the
corresponding directed path in $\Zdo$. Let also $\{V(x)\}_{x\in\Zdo}$ be a
collection of i.i.d.\ random variables with finite exponential moments, whose
joint law is denoted by $\bbP$. One is then interested in the behaviour of the
path $\gamma$ under the random probability measure
\[
\mu_n^\omega(\gamma) = (Z_{n;\beta}^\omega)^{-1}\, \exp\bigl(-\beta \sum_{k=1}^n
V(\gamma_k)\bigr)\, (2d)^{-n},
\]
where $\beta\geq 0$ is the inverse temperature. The behaviour of the path
$\gamma$ is closely related to the behaviour of the partition function
$Z_{n;\beta}^\omega$. Namely, one distinguishes between two regimes: the
weak disorder regime, in which $\lim_{n\to\infty}
Z_{n;\beta}^\omega/\bbE(Z_{n;\beta}^\omega) > 0$, $\bbP$-a.s., and the strong
disorder regime, in which this limit is zero. It is known~\cite{CoYo06} that
there is a sharp transition between these two regimes at an inverse temperature
$\beta_c$ which is non-trivial when $d\geq 3$. In the weak disorder regime 
({$\beta <\beta_c$}), the
path $\gamma$ behaves diffusively, in that $\gamma_n$ satisfies a CLT.
Diffusivity at sufficiently small values of $\beta$ was first established
in~\cite{ImSp88}; this was extended to an almost-sure CLT in~\cite{Bo89}; a CLT
(in probability) valid in the whole weak disorder regime was then obtained
in~\cite{CoYo06}.  

In dimensions $d\geq 3$ the sequence 
$Z_{n;\beta}^\omega/\bbE(Z_{n;\beta}^\omega)$
is bounded in $\bbL_2$ for all sufficiently small values of $\beta$. In such a
situation local limit versions of the CLT, which hold in probability, were
established in~\cite{Sinai,Vargas}. 

In the case of directed polymers the disorder is always strong in dimensions
$d=1,2$~\cite{CV06, Lacoin} and at sufficiently low temperatures. Concerning the
(nondiffusive) behaviour in the strong disorder regime, we refer the reader
to~\cite{CC11} and references therein.

In this work, we consider diffusive behaviour in dimensions $d+1\geq 4$ for
the related models of \emph{stretched polymers}.
{The choice of notation $d+1$ indicates that stretched polymers on $\Zdo$ should be 
compared with directed polymers in $d$ dimensions. }
However, a stretched 
path $\gamma$ can be any nearest-neighbour path on $\Zdo$, which is permitted to
bend and to return to particular vertices an arbitrary number of times. The
disorder is modelled  by a collection $\lbr V(x)\rbr_{x\in\Zdo}$ of i.i.d.\
non-negative random variables. Each visit of the path to a vertex $x$ exerts
the price ${\rm e}^{-\beta V (x)}$. The \emph{stretch} is introduced in
one of the following two natural ways:
\begin{itemize}
\item The path $\gamma$ starts at $0$ and ends at a hyperplane at distance $n$
from $0$ and has arbitrary length. This is a model of crossing random walks in
random potentials. In dimension $d+1=2$, it presumably provides a better
approximation to Ising interfaces in the presence of random impurities. 
\item The path $\gamma$ has a fixed length $n$, but it is subject to a drift,
which can be interpreted physically as the effect of a force acting on the
polymer's free end. 
\end{itemize}
The precise model is described below. At this stage let us remark that models of
stretched polymers have a richer morphology than models of directed polymers.
Even the issue of ballistic behaviour for annealed models is 
non-trivial~\cite{IoffeVelenik-Annealed, KM11, IVCritical}. The issue of
ballistic behaviour in the quenched case is still not resolved completely, and,
in order to ensure ballisticity one needs to assume that the random potential
$V$ is strictly positive in the crossing case, and that the applied drift is 
sufficiently large in the fixed length case. Both conditions are designed to
ensure a somewhat massive nature  of the model.

As in the directed case, the disorder is always strong~\cite{Zy-10} in low
dimensions $d+1= 2, 3$ or at sufficiently low temperatures. 

In the case of higher dimensions $d+1\geq 4$, the existence of weak disorder on
the level of equality between quenched and annealed free energies was
established in~\cite{Flury,Zygouras}. 
The case of high temperature discrete Wiener sausage with drift was addressed 
in~\cite{ErwinStudent}.

In the crossing case, a CLT in probability was established in~\cite{IV-Crossing}
in all dimensions $d+1\geq 4$ at sufficiently high temperatures.

The aim of the present paper is to establish an almost-sure CLT for the endpoint
of the fixed-length version of the model of stretched polymers with non-zero
drifts, also at sufficiently high temperatures and in all dimensions $d+1\geq
4$.

\subsection{Class of Models}
\noindent
\textbf{Polymers.}
For the purpose of this paper, a polymer $\gamma = (\gamma_0, \dots, \gamma_n)$ is a nearest-neighbour trajectory on the
integer lattice $\Zdo$. Unless stressed otherwise, $\gamma_0$ is always placed at the origin. The length of the polymer
is $ \abs{\gamma}\df n$ and its spatial extension is $X (\gamma ) \df \gamma_n - \gamma_0$.  In the most general case,
neither the length nor the spatial extension are fixed.
\smallskip

\noindent
\textbf{Random Environment.} The random environment is a collection $\lbr V(x)\rbr_{x\in\Zdo}$ of non-degenerate
non-negative i.i.d.\ random variables which are normalised by $0\in {\rm supp}(V)$. There is no moment assumptions on
$V$. The case of traps, $p_\infty \df \bbP\lb V=\infty \rb >0$, is  not excluded, but then we shall assume that
$p_\infty$ is small enough.  In particular, we shall assume that $\bbP$-a.s.\ there is an  
infinite connected cluster
${\rm Cl}_\infty (V)$ of the set $\lbr x : V (x) <\infty\rbr$ in $\Zdo$. In fact, we shall 
assume more: Given
$\bbR^{d+1}\ni h\neq 0$ and a number $\delta \in  {(0, \frac{1}{\sqrt{d+1}} )}$, define 
the positive cone
\be
\label{eq:dhcone}
\calY^h_\delta  \df \setof{x\in\bbR^{d+1}}{x\cdot h\geq \delta\abs{x}\abs{h}} .
\ee
By construction, the cones $\calY^h_\delta$ always contain at least one lattice direction $\pm \sfe_i ,\ i=1, \dots,
d+1$.  We assume that  it is possible to choose $\delta$  in such a fashion that, for any $h$, the intersection ${\rm
Cl}_\infty^{h, \delta}(V)\df  {\rm Cl}_\infty (V)\cap\calY^h_\delta$ contains ($\bbP$-a.s.) an infinite connected
component. For the rest of the paper, we fix such a $\delta\in 
{(0, \frac{1}{\sqrt{d+1}} )}
$ and use the reduced notation
$\calY^h$ and ${\rm Cl}_\infty^{h}(V)$ for the corresponding
cones~\eqref{eq:dhcone} and percolation clusters.

\smallskip
\noindent
\textbf{Weights and Path Measures.} The reference measure $\sfp (\gamma ) \df (2(d+1))^{-\abs{\gamma}}$ is given by
simple random walk weights. The
polymer weights we are going to consider are quantified by
two parameters: the inverse temperature $\beta\geq 0$ and the external pulling
force $h\in\bbR^{d+1}$.

\noindent
The random quenched weights are given by
\be
\label{eq:qweight}
q_{h,\beta }^\omega (\gamma) \df \exp\Bigl\{  h\cdot X (\gamma)
- \beta\sum_{1}^{\abs{\gamma}} V (\gamma_i)\Bigr\}  \sfp (\gamma).
\ee
The corresponding deterministic annealed weights are given by
\be
\label{eq:aweight}
q_{h,\beta} (\gamma) \df
\bbE q_{h,\beta}^\omega (\gamma)  =
\exp\lbr h\cdot X (\gamma)
-
\Phi_\beta (\gamma)
\rbr \sfp (\gamma) ,
\ee
where $\Phi_\beta (\gamma) \df \sum_x \phi_\beta \bigl(\ell_\gamma(x) \bigr)$, with $\ell_\gamma(x)$ denoting the local
time (number of visits) of $\gamma$ at $x$, and
\begin{equation}
\label{eq:phibeta}
\phi_\beta (\ell) \df -\log\bbE {\rm e}^{-\beta \ell V} .
\end{equation}
Note that the annealed potential is positive, non-decreasing and 
attractive, in the sense that
\be
\label{eq:attractive}
0 < \phi_\beta (\ell ) \leq 
 \phi_\beta (\ell +m ) \leq \phi_\beta (\ell ) +\phi_\beta (m ),\qquad\forall\, 
\ell,m\in\bbN .
\ee
\smallskip
\noindent
In the sequel, we shall drop the index $\beta$ from the notation, and we shall
drop the index $h$ whenever it equals zero. With this convention, the quenched
partition functions are defined by
\begin{equation}
\label{eq:pf}
Q_n^\omega (x)  \df \sumtwo{X (\gamma)=x}{|\gamma |=n} q^\omega (\gamma) , \ \
Q_n^\omega (h ) \df \sum_{|\gamma | = n} q_h^\omega  (\gamma ) =
\sum_x {\rm e}^{h\cdot x} Q_n^\omega (x) ,
\end{equation}
and we use $Q_n (x) \df \bbE Q_n^\omega (x)$
and $Q_n (h) \df \bbE Q_n^\omega (h)$ to denote their annealed counterparts.

Finally, we define the corresponding quenched and annealed path measures by
\begin{equation}
\label{eq:nMeasures}
\bbQ_{n, h}^\omega  (\gamma ) \df \One_{\lbr \abs{\gamma} = n\rbr} 
\frac{q_h^\omega (\gamma )}{Q_n^\omega (h)}\quad
\text{and}\quad \bbQ_{n,h}  (\gamma ) \df \One_{\lbr \abs{\gamma} = n\rbr} \frac{q_h (\gamma )}{Q_n (h)} .
\end{equation}
\noindent
\textbf{Very Weak Disorder.}
The notion of very weak disorder is {technical} and 
it depends on the strength $|h|$ of the 
pulling force , dimension $d\geq 3$ and 
the distribution of $V$. 
{
By Lemma~\ref{lem:WDBound} below, there exists a function $\zeta_d$ on 
$(0,\infty )$ such that a certain $\bbL_2$-estimate \eqref{eq:WDBound}
holds if $\phi_\beta (1) < \zeta_d (| h| )$. 
}
\begin{definition}
The model of stretched polymers is in the regime of very weak disorder if 
$d\geq 3$ and 
\begin{equation}
\label{eq:veryweak}
\phi_\beta (1) < \zeta_d (| h| ) .
\end{equation}
\end{definition}

\subsection{The Result}
\label{sub:result}
Fix $h\neq 0$. Then~\cite{Zerner, Flury-LD, IoffeVelenik-Annealed}
\be
 \label{eq:lambda}
\lambda = \lambda (\beta ,h ) \df   \lim_{n\to\infty}\frac1n \log Q_n (h) \in
(0,\infty ) ,
\ee
for all sufficiently small $\beta$. The following two quantities play a
central role in our limit theorems:
\[
v = v (h,\beta) \df \nabla\lambda (h),\qquad
\Sigma \df {\rm Hess }[\lambda] (h).
\]
 {If $\beta$ is sufficiently small then  
$v\neq 0$ and the
matrix $\Sigma$ is positive definite and, moreover, 
 $v$ and $\Sigma$ are the limiting spatial extension and, 
respectively, the diffusivity matrix for the annealed model. (Sections~4.1,4.2 
 in \cite{IoffeVelenik-Annealed}).
In Subsection~\ref{sub:Three} we recall further relevant facts about the annealed model. }
\begin{bigthm}
 \label{thm:A}
Fix $h\neq 0$. Then, in the regime of very weak disorder,
the following holds $\bbP$-a.s.\ on the event $\lbr 0\in {\rm Cl}_\infty
(V)\rbr$:
\begin{itemize}
\item  The limit
\be
\label{eq:Aclaima}
 \lim_{n\to\infty} \frac{Q_n^\omega  (h)}{Q_n (h)}
\ee
exists and is a strictly positive, square-integrable random variable.
\item 
{There exists a sequence $\{\epsilon_n\}$ with $\lim\epsilon_n = 0$, 
such that 
\be
\label{eq:Aclaimb}
\sum_n \bbQ_{n,h}^\omega \Bigl( \bigl| \frac{ X(\gamma)}{n} - v \bigr| 
> \epsilon_n \Bigr) < \infty .
\ee
}
\item 
For every $\alpha\in \bbR^{d+1}$,
\be
\label{eq:Aclaimc}
 \lim_{n\to\infty} \bbQ_{n, {h}}^\omega \Bigl( \exp
\bigl\{ \frac{i\alpha}{\sqrt{n}} (X(\gamma )-nv) \bigr\}\Bigr) =
\exp\bigl\{ -\tfrac12 \Sigma\alpha\cdot\alpha \bigr\} .
\ee
\end{itemize}
\end{bigthm}
{We would like to stress  that, in contrast to the case of directed 
polymers~\cite{CoYo06}, our CLT does not pertain to the whole of the weak 
disorder region. 
The procedure of first fixing $h\neq 0$ and then going to $\beta >0$ sufficiently small is essential.
Furthermore, even in the regime we are working with, \eqref{eq:Aclaimc} 
 falls short of the local CLT form of results as developed for directed polymers in \cite{Vargas}.
 These and related issues remain open in the context of stretched polymers.}  

{
Few remarks on the history of the problem:}
  Flury \cite{Flury} had established that under the conditions of
Theorem~\ref{thm:A} (and some additional moment
assumptions of the potential $V$)
\be
\label{eq:FZ}
\lim_{n\to\infty} \frac1n\log  \frac{Q_n^\omega  (h)}{Q_n (h)} = 0
\ee
for on-axis exterior forces $h$. \eqref{eq:FZ} was then extended to arbitrary
directions $h\in\bbR^{d+1}$ by Zygouras~\cite{Zygouras}. In~\cite{Flury}, the
analysis was carried out directly in the canonical ensemble of polymers with
fixed length $n$.  In~\cite{Zygouras}, the author derives results for the
conjugate ensemble of the so-called crossing
random walks.

Large deviations {(LD)} under both $\bbQ_{n,h}$ and $\bbQ_{n,h}^\omega$ 
were investigated in~\cite{Zerner, Flury-LD}. The results therein
imply that, under the conditions of Theorem~\ref{thm:A}, the model is ballistic in the sense that the value of the
quenched rate function at zero is strictly positive. However, \cite{Zerner, Flury-LD} do not imply a 
{law of large numbers}
(LLN) even in the
annealed case. In particular, these works do not contain information on the strict convexity of the corresponding rate
functions. The issue of strict convexity for the annealed rate functions was settled in~\cite{IoffeVelenik-Annealed}.
Therefore, \eqref{eq:Aclaimb} is a direct consequence of~\eqref{eq:FZ} and of the analysis of annealed canonical
measures in~\cite{IoffeVelenik-Annealed}.

The main new results  of this work are~\eqref{eq:Aclaima} and~\eqref{eq:Aclaimc}. A version of Theorem~\ref{thm:A} for
the ensemble of crossing random walks appears in~\cite{IV-Crossing}. The length
of crossing random walks is not fixed (only suppressed by an additional positive
mass), and they are required to have their second endpoint on a distant
hyperplane.  In this way, crossing random walks in random potential are much more ``martingale''-like 
than canonical
random walks. Moreover, the canonical constraint of fixed length does not facilitate computations, 
to say the
least.  
Finally, the CLT of~\cite{IV-Crossing} was only established in probability
and not $\bbP$-a.s.
Thus, although the techniques developed in~\cite{IV-Crossing} are useful here, they certainly do not imply
the claims of Theorem~\ref{thm:A}, and an alternative approach was required.

\subsection{Irreducible Decomposition, {Basic Ensembles} and Basic Partition Functions}

A polymer $\gamma = (\gamma_0 , \dots , \gamma_n )$ is said to be cone-confined if
\be
\label{eq:cconfined}
\gamma\subset \lb \gamma_0 +\calY^h \rb \cap\lb \gamma_n - \calY^h\rb .
\ee
A cone-confined polymer which cannot be represented as the concatenation of two (non-singleton) cone-confined polymers
is said to be irreducible. We denote by $\calT (x)$ the collection of all cone-confined paths 
leading from $0$ to $x$,
and by $\calF (x)\subset\calT (x)$ the set of irreducible cone-confined paths.
{In the sequel we shall refer to $\calF (x)$ and $\calT (x)$ 
as to basic ensembles}. 
The basic partition functions are defined by
\be
\label{eq:weights}
 \tq{x,n} \df {\rm e}^{-\lambda n}
\sum_{\gamma\in\calT (x)}\One_{\lbr \abs{\gamma} =n\rbr} q_{h}^\omega (\gamma)
\quad\text{and}\quad
\fq{x ,n} \df {\rm e}^{-\lambda n} \sum_{\gamma\in{\calF (x)}}\One_{\lbr \abs{\gamma} =n\rbr} q_{h}^\omega  (\gamma) .
\ee
We also set, accordingly, $\tq{n}\df \sum_x\tq{x, n}$ and $\fq{n}\df \sum_x\fq{x,n}$. 
The annealed counterparts of all
these quantities are denoted by $\ta{x,n} \df \bbE\tq{x, n}$, $\fa{x,n } 
\df \bbE\fq{x, n}$, $\ta{n}\df\bbE\tq{n}$ and
$\fa{n} \df \bbE\fq{n}$.
As shown in {Section~3.6 of} \cite{IoffeVelenik-Annealed}, the collection $\lbr \fa{x,n}\rbr$
forms a probability distribution,
\[
\sum_{n}\sum_{x}\fa{x, n} =  \sum_n \fa{n} = 1 ,
\]
with exponentially decaying tails:
\be
\label{eq:exp-tails}
\sum_{m\geq n} \fa{m }\df \sum_{m\geq n}\sum_x \fa{x ,m} \leq {\rm e}^{- \nu n} ,
\ee
where $\nu = \nu (\beta , h)\to\infty$ as $\beta$ becomes large, {and $\inf_{\beta\geq 0} \nu(\beta,h)>0$, for
all $h\neq 0$}.
{
\begin{remark}
\label{rem:tn}
Since by definition  polymers are nearest neighbour paths, it always holds that 
$\ta{x,n} = \ta{x,n} \One_{\lbr \abs{x}\leq n\rbr}$.
\end{remark}
}
\smallskip

As in~\cite[Subsections~2.7 and 3.5]{IV-Crossing}, the following statement about basic ensembles implies the
claims~\eqref{eq:Aclaima} and~\eqref{eq:Aclaimc}  of Theorem~\ref{thm:A}:
\begin{bigthm}
\label{thm:B}
Fix $h\neq 0$. Then, in the regime of very weak disorder,
the following holds $\bbP$-a.s.\ on the event $\lbr 0\in {\rm Cl}_\infty^{h}
(V)\rbr$:
\begin{itemize}
\item The limit
\be
\label{eq:claima}
s^\omega \df \lim_{n\to\infty} \frac{\tq{n}}{\ta{n}}
\ee
exists and is a strictly positive, square-integrable random variable.
\item 
For every $\alpha\in\bbR^{d+1}$,
\be
\label{eq:claimc}
\lim_{n\to\infty}
\frac1{\tq{n}  }\sum_x\exp\bigl\{ \frac{i\alpha}{\sqrt{n}}\cdot (x-nv) 
\bigr\} \tq{x ,n}
=
\exp\bigl\{ -\tfrac12 \Sigma\alpha\cdot\alpha \bigr\} .
\ee
\end{itemize}
\end{bigthm}
For the rest of the paper, we shall focus on the proof of Theorem~\ref{thm:B}.

\section{Proof of Theorem~\ref{thm:B}}
To facilitate the exposition, we shall  consider the case of on-axis external 
force $h = h\sfe_1$.  The proof, however,
readily applies for any non-zero $h\in\bbR^{d+1}$.
{By lattice symmetries, the mean displacement $v =
\nabla\lambda (h)$ lies along the direction $\sfe_1$;  
 $v = v\sfe_1$. As it was already mentioned in the beginning of Subsection~\ref{sub:result}, 
$v\neq 0$ whenever $\beta$ is small enough. We proceed assuming that both the drift and the
speed are positive $h, v >0$.
}
\subsection{Three Main Inputs}
\label{sub:Three}
The reduction to basic ensembles constitutes the central step of the Ornstein-Zernike theory. We rely on three facts:
The first is the refined description of the annealed phase in the ballistic regime (which, in our regime, will always
correspond to first fixing $h\neq 0$ and then choosing $\beta>0$ small enough). Below, we shall summarize the required
results from~\cite{IoffeVelenik-Annealed, IV-Erwin}.  The second is an $\bbL_2$-type estimate on overlaps which holds
for all $\beta$ sufficiently small, and which could be understood as quantifying the notion of very weak disorder we
employ here. The third  is a maximal inequality for the so-called {\em mixingales}, due to McLeish. Unlike directed
polymers, stretched polymers do not possess natural martingale structures, and McLeish's result happens to provide a
convenient alternative framework.
\paragraph{Ornstein-Zernike theory of annealed models.}
Annealed asymptotics of $\ta{n}$  in the ballistic regime are not related to the strength of disorder and hold for all
values of $\beta\geq 0$ and appropriately large drifts $h$ . In particular, for each $h\neq 0$ fixed, the annealed model
is ballistic for all sufficiently small $\beta$. We refer 
to~\cite[{Sections~4.1 and~4.2}]{IoffeVelenik-Annealed} and 
to~\cite[{Section~4.2}]{IV-Erwin} for the proof of the
following:  Fix $h\neq 0$; then, for all $\beta >0$ small enough, $\lambda (h)
>0${, $\nabla\lambda (h)\neq 0$ and ${\rm Hess}[\lambda] (h)$ is positive
definite}. Furthermore, there exist a small
complex neighbourhood $\calU\subset\bbC^{d+1}$ of the origin, an analytic function $\mu$ (with $\mu (0) = 0$) on $\calU$
and a non-vanishing analytic function $\kappa\neq 0$ on $\calU$ such that:
\be
\label{eq:tanAsympt}
\lim_{n\to\infty} {\rm e}^{-n\mu (z)}  \ta{n}(z) \df
\lim_{n\to\infty} {\rm e}^{-n\mu (z)} \sum_x\ta{x ,n}{\rm e}^{z\cdot x}
= \frac{1}{\kappa (z)} ,
\ee
uniformly exponentially fast on $\calU$.
Note~\cite{IV-Erwin} ({Section~4.2})  that $\lambda (h +z) = \lambda (h) +\mu (z)$
for real $z${, and thus $v = \nabla\lambda (h) = \nabla\mu(0)$ and $\Sigma
= {\rm Hess }[\lambda] (h) = {\rm Hess}[\mu] (0)$}.

The annealed model satisfies a local LD upper bound: There exists $ c = c
(\beta,h)>0$ such that, for all $x\in\calY^h$,
\be
\label{eq:Aupper}
\ta{x, n} 
\leq 
\frac{1}{c \sqrt{n^{d+1}}}
\exp\Bigl\{ -c\frac{|x- nv|^2}{n}\Bigr\} .
\ee
{
In view of Remark~\ref{rem:tn} the above bound is trivial whenever $\abs{x} >n$. 
}
\newline
Finally, it is a straightforward consequence of~\eqref{eq:tanAsympt} that the
following annealed CLT holds:
\be
\label{eq:AnCLT}
\aS{n} \bigl(\frac{\alpha}{\sqrt{n}}\bigr) \df
\sum_{x} \ta{x ,n} \exp\bigl\{ i  \frac{\alpha}{\sqrt{n}} \cdot (x - nv) 
\bigr\}  =
\frac{1}{\kappa (0) }
\exp\bigl\{ -\tfrac12\Sigma\,\alpha\cdot\alpha\bigr\} \bigl(1 + O(n^{-1/2}) 
\bigr) ,
\ee
with the second asymptotic equality holding uniformly in $\alpha $ on compact
subsets of $\bbR^{d+1}$.

\paragraph{An $\bbL_2$-estimate.}
Fix an external force $h\neq 0$. 
We continue to employ notation $v = v ( h,\beta )$.
For a subset $A\subseteq\bbZ^{d+1}$, let $\calA$ be the $\sigma$-algebra
generated by
$\lbr V (x)\rbr_{x\in A}$. We shall call such $\sigma$-algebras cylindrical.
\begin{lemma}
\label{lem:WDBound}
{For any dimension $d\geq 3$ there exist a positive non-decreasing function 
$\zeta_d$ on $(0,\infty )$ and a number $\rho < 1/12$ such that the following 
holds: If $\phi_\beta (1) <\zeta_d (\abs{h})$, then}
there exist  constants $c_1 ,c_2 <\infty$ such that the random weights~\eqref{eq:weights}
satisfy: 
\be
\label{eq:WDBound}
\begin{split}
&\abs{
\bbE 
 \bigl[ \tq{x , \ell}\tq{\xpr ,\ell } \bbE(\fxq{y ,m } -\fa{y, m}\,|\,
\calA) \bbE (\fxpq{\ypr ,\mpr} -\fa{\ypr ,\mpr} \,|\,\calA)
 \bigr]
} \\
&\quad
\leq \frac{c_1{\rm e}^{- c_2 (m+\mpr )}}{ \ell^{d+1 -\rho}} \exp\Bigl\{
-c_2 \Bigl( \abs{x -\xpr}
+ \frac{\abs{x -\ell v}^2}{\ell} +\frac{\abs{\xpr -\ell 
v}^2}{\ell}\Bigr)\Bigr\} , 
\end{split}
\ee
{for all} $x , \xpr, m, \mpr, y,
\ypr , \ell $ and  all cylindrical  $\sigma$-algebras $\calA$ such that both
$\tq{x,\ell}$ and $\tq{\xpr,\ell}$ are
$\calA$-measurable.
\end{lemma}
\begin{remark}
{ 
The above bound is non-trivial only if both 
$\abs{x}, \abs{\xpr} \leq \ell$ (Remark~\ref{rem:tn}).} Also, 
there is nothing sacred about the condition $\rho <1/12$. We just need $\rho$ to
be sufficiently small. In fact, \eqref{eq:WDBound} holds with  $\rho=0$,
although a proof of such statement would be a bit  more involved. 
\end{remark}

In spite of its technical appearance, \eqref{eq:WDBound} has a transparent
intuitive meaning: For $\rho=0$, the expressions on the
right-hand side are just local limit bounds for a couple of independent annealed
polymers  with exponential penalty for
disagreement at  their end-points.  The irreducible terms have exponential
decay. In the very weak disorder regime, the
interaction between polymers does not destroy these asymptotics.
The proof of Lemma~\ref{lem:WDBound} is relegated to the concluding
Section~\ref{sec:WDBound}.
\smallskip

\paragraph{McLeish's Maximal Inequality.}
Let $Z_1 , Z_2, \ldots $ be a sequence of zero-mean, square-integrable random
variables. Let also
$\lbr\calA_k\rbr_{-\infty}^{\infty}$ be a filtration of $\sigma$-algebras.
Suppose that we have chosen $\epsilon >0$ and
numbers $\psi_1, \psi_2 , \ldots $ in such a way that
\be
\label{eq:McLeish}
\bbE\bigl[\bbE(Z_\ell \,|\, \calA_{\ell-k})^2\bigr] \leq
\frac{\psi_\ell^2}{(1+k)^{1+\epsilon}}
\quad\text{and}\quad
\bbE\bigl[\bigl( Z_\ell - \bbE(Z_\ell \,|\, \calA_{\ell+k})\bigr)^2\bigr] \leq
\frac{\psi_\ell^2}{(1+k)^{1+\epsilon}}
\ee
for all $\ell = 1,2, \ldots $ and $k\geq 0$. Then~\cite{McLeish} there exists $K
= K (\epsilon) < \infty$ such that, for
all $n_1 \leq  n_2$,
\be
\label{eq:Maximal}
\bbE\Bigl[\max_{n_1\leq r\leq n_2}\Bigl( \sum_{n_1}^r Z_\ell\Bigr)^2\,\Bigr] 
\leq K \sum_{n_1}^{n_2}
\psi_\ell^2 .
\ee
{
\begin{remark}
 \label{re:conv} 
In particular, if $\sum_\ell \psi_\ell^2 <\infty$, then $\sum_\ell Z_\ell$
converges $\bbP$-a.s.\ and in $\bbL_2$.
\end{remark}
}
\smallskip

In the sequel, we shall always work with the following filtration
$\lbr\calA_m\rbr$. Recall that we are discussing
on-axis {positive} drifts $h = h\sfe_1$ {which, for small $\beta$, give rise to on-axis 
limiting spatial extension $v = v\sfe_1$ with $v >0$.} 
At this stage, define
the hyperplanes $\cHm{m}$ and the corresponding $\sigma$-algebras $\calA_m$ as
\be
\label{eq:sigma-algebras}
\cHm{m} = \setof{x\in\bbZ^{d+1}}{x\cdot\sfe_1 \leq m|v|}\quad\text{and}\quad
\calA_m =\sigma\setof{V (x)}{x\in\cHm{m}} .
\ee

\paragraph{Notation for asymptotic relations.}
The following notation is convenient, and we shall use it throughout the text: 
Given a (countable) set of indices
${\calI}$ and two positive sequences $\lbr a_\alpha , b_\alpha
\rbr_{\alpha\in\calI}$, we say that $a_\alpha \leqs
b_\alpha$ if there exists a constant $c>0$ such that $a_\alpha \leq  c b_\alpha$
for all $\alpha\in\calI$ . We shall use
$a_\alpha \cong b_\alpha$ if both $a_\alpha\leqs b_\alpha$ and $a_\alpha\geqs
b_\alpha$ hold. For instance, for any
$\epsilon >0 $ fixed,
\begin{equation}
\label{eq:m-lbound}
\frac{{\rm e}^{-  c_3 k^2 /\ell}}{\ell^{{(1+\epsilon )/2}} }\leqs
\frac{1}{(1+k)^{1+\epsilon}} ,
\end{equation}
where the index set $\calI$ is the set of pairs of integers $(k , \ell )$ with
$k\geq 0$ and $\ell >0$.

\paragraph{Structure of upper bounds.}

Our upper bounds are based on~\eqref{eq:m-lbound}, \eqref{eq:WDBound} (applied
with $\rho = \epsilon/2$) and on~\eqref{eq:Maximal}. Recall that $\rho <1/12$,
and hence $\epsilon <1/6$. 
\newline
In the sequel, we shall repeatedly derive variance bounds on quantities of the
type $\sum_{\ell\leq n}Z^{(n)}_\ell$. The
most general form of $Z^{(n)}_\ell$ we shall consider is
\be
\label{eq:Znl}
Z^{(n)}_\ell = \sum_{x}\tq{x,\ell}\sum_{y,m} a^{(n)}_{x,\ell} (y,m)\,
\bigl(\fxq{y,m} - \fa{y,m}\bigr) ,
\ee
where $\bigl\{ a^{(n)}_{x,\ell} (y ,m )\bigr\}$ are 
{arrays of real or
complex numbers}. Assume that there exists another family of 
{(non-negative)} 
arrays $\bigl\{ \hat
a^{(n)}_{x,\ell}\bigr\}$ {and a number $\nu >0$}  such that
\be
\label{eq:abar}
{\rm e}^{-c_2 m} \sum_{|y|\leq m}\big|a^{(n)}_{x,\ell} (y,m)\bigr| \leqs
{{\rm e}^{-\nu m}} \hat a^{(n)}_{x,\ell} ,
\ee
where the constant $c_2$ is inherited from~\eqref{eq:WDBound}.
{
\begin{lemma}
\label{lem:LemUB-1}
Set $\epsilon = \frac{\rho}{2}$, where $\rho $ is the power which 
shows up  in \eqref{eq:WDBound}.    
Under assumption \eqref{eq:abar} 
\be
\label{eq:GUpper1}
\bbE\bigl[\bigl( \bbE(Z^{(n)}_\ell \,|\, \calA_{\ell-k}) \bigr)^2\bigr]
\leqs  \frac{1}{\ell^{d+1 - \epsilon/2 }}
\sum_{x\in\calH^-_{\ell -k}}
{\rm e}^{-c_2\frac{\abs {x-\ell v}^2}{\ell}}\bigl( \hat 
a^{(n)}_{x,\ell}\bigr)^2 , 
\ee
and 
\be 
\label{eq:GUpper1-plus}
\bbE\bigl[\bigl(
Z^{(n)}_\ell - \bbE(Z^{(n)}_\ell \,|\, \calA_{\ell+k}) \bigr)^2 \bigr]
\leqs 
\frac{1}{\ell^{d+1 - \epsilon/2 }}
\sum_{x}
{\rm e}^{-c_2\frac{\abs {x-\ell v}^2}{\ell}- \nu {\rm d}_{\ell +k } (x)}
\bigl(\hat a^{(n)}_{x,\ell}\bigr)^2 . 
\ee
Above we introduced a 
provisional notation ${\rm d}_{r} (x) \df \lb r \abs{v} - \sfe_1\cdot x\rb\vee 0$ 
for the distance from $x$ to $\calH_r^+=\bbZ^{d+1}\setminus\calH_r^-$.
\end{lemma}
}
{
\begin{proof}
Since $\bbE \lb  \tq{x,\ell} 
\lb \fxq{y,m} - \fa{y,m}\rb  |\, \calA_{\ell-k}\rb = 0$ whever $ x\not\in \calH_{\ell-k}^-$, 
\be
\label{eq:Znone}
 \bbE(Z^{(n)}_\ell \,|\, \calA_{\ell-k} )
= \sum_{x\in \calH_{\ell-k}^-} \tq{x,\ell}\sum_{y,m} a^{(n)}_{x,\ell} (y,m)\,
\bbE(\fxq{y,m} - \fa{y,m} \,|\, \calA_{\ell-k} ) .
\ee
Taking the expectation of the square of the latter expression
and, for each $x, \xpr$,  factorizing replicas using $\abs{ab}\leq \frac{a^2 +b^2}{2}$, one
derives the first inequality~\eqref{eq:GUpper1} directly from 
\eqref{eq:WDBound} and 
\eqref{eq:abar}. \newline
Next, 
\begin{multline}
\label{eq:Zntwo}
Z^{(n)}_\ell - \bbE(Z^{(n)}_\ell \,|\, \calA_{\ell+ k})
=
\sum_{x\in\calH_{\ell+k}^+} \tq{x,\ell}
\sum_{y,m} a^{(n)}_{x,\ell} (y,m)\,
\bigl( \fxq{y,m} - \fa{y,m}\bigr)
\\
+
\sum_{x\in\calH_{\ell+k}^-}
\tq{x,\ell}
\sum_{z\in\calH_{\ell+k}^+}\sum_m
a^{(n)}_{x,\ell} (z-x,m)\,
\bigl(
\fxq{z-x,m} - \bbE(\fxq{z-x,m} \,|\, \calA_{\ell+k}) \bigr) .
\end{multline}
For any $x\in \calH_{\ell +k}^+$, ${\rm d}_{\ell +k}  (x ) = 0$, 
and   the first term in~\eqref{eq:Zntwo} has exactly the same structure as the
right-hand side of~\eqref{eq:Znone}.  
On the
other hand, if $x \in\calH_{\ell +k}^-$ 
and $z\in\calH_{\ell
+k}^+$, then, in view of Remark~\ref{rem:tn},  $\fxq{z-x,m}$ can be different
from zero only if
$m\geq {\rm d}_{\ell +k } (x)$ and $\abs{z-x}\leq m$. 
 Therefore, \eqref{eq:GUpper1-plus} is also a direct consequence of 
\eqref{eq:WDBound} and \eqref{eq:abar}. 
\end{proof}
}
{
The following is a useful corollary:
\begin{lemma}
\label{lem:LemUB-2}
If $\hat a^{(n)}_{x,\ell}\leqs \hat a^{(n)}_{\ell}$, then the  bounds
\eqref{eq:GUpper1} and 
\eqref{eq:GUpper1-plus}
reduce to
\be
\label{eq:GUpper2}
\bbE\bigl[\bigl( \bbE(Z^{(n)}_\ell \,|\, \calA_{\ell-k}) 
\bigr)^2\bigr],\; 
\bbE\bigl[\bigl(
Z^{(n)}_\ell - \bbE(Z^{(n)}_\ell \,|\, \calA_{\ell+k})\bigr)^2\bigr]
\leqs
\bigl(\hat a^{(n)}_{\ell}\bigr)^2\frac{1}{\ell^{d/2 -\epsilon}
(1+k)^{1+\epsilon}} .
\ee
\end{lemma}
}
\begin{proof}
{Consider first the right-hand side of \eqref{eq:GUpper1}.}
Since $\sum_{x\in\calH^-_{\ell -k}} {\rm e}^{-c_2\frac{\abs {x-\ell
v}^2}{\ell}}\leqs
\ell^{\frac{d+1}{2}}$, the non-trivial part  is to
check~\eqref{eq:GUpper2} for large values of $k$. In the
latter case, we may assume that $\abs{x-v\ell} >\frac{k\abs{v}}{2}$ for all
$x\in\cHm{\ell - k}$. 
Consequently, 
the sum on the right-hand side of~\eqref{eq:GUpper1} is bounded above by
\be
\label{eq:VarXlm}
\begin{split}
 &
\sum_{x\in \cHm{ \ell - k}}
{\rm e}^{-c_2 \abs{x-v\ell}^2/\ell}
\leqs 
\int_{\abs{y} >\frac{k \abs{v}}{2}} {\rm e}^{- c_2\abs{y}^2/\ell}\dd y\\
&{\eqvs} 
\int_{\frac{k\abs{v}}{2}}^\infty r^d {\rm e}^{-c_2 r^2 /\ell}\dd r \leqs
\ell^{(d+1)/2}{{\rm e}^{-c_3 k^2 /\ell}}
\leqs \frac{\ell^{d/2 +1  + \epsilon/2} }{(1+k)^{1+\epsilon}} ,
\end{split}
\ee
the last inequality being an application of~\eqref{eq:m-lbound}. 
\eqref{eq:GUpper2} follows.\newline
{ 
Turning to the right-hand side of \eqref{eq:GUpper1-plus}, we see that 
it  remains to derive an upper bound on 
\be 
\label{eq:extra-term}
 \sum_{x\in\calH_{\ell +k}^- }
{\rm e}^{-c_2\frac{\abs {x-\ell v}^2}{\ell}- \nu {\rm d}_{\ell +k } (x)} 
\leqs \sum_{\abs{y} >\frac{k \abs{v}}{2}} {\rm e}^{-c_2\frac{\abs{y}^2}{\ell}}
+ \sum_{\abs{y} \leq \frac{k \abs{v}}{2}} {\rm e}^{- \nu  {\rm d}_{k } (y)} .
\ee
The first sum above is treated as  in  \eqref{eq:VarXlm}. 
On the other hand, the second sum is bounded above as $\leqs {\rm 
e}^{-\nu^\prime k}$, uniformly in all $k$ sufficiently large. Since ${\rm 
e}^{-\nu^\prime k} \leqs (1+k)^{-1-\epsilon}$, 
the bound \eqref{eq:GUpper2} for 
$\bbE\bigl[\bigl(
Z^{(n)}_\ell - \bbE(Z^{(n)}_\ell \,|\, \calA_{\ell+k}) \bigr)^2\bigr]$ follows 
as well. 
}
\end{proof}
As an application  of~\eqref{eq:GUpper2} we derive the following convergence result:
\begin{lemma}
\label{lem:largem}
{Assume that, for some $\nu^\prime>0$, the asymptotic 
bound~\eqref{eq:GUpper2} is, uniformly in $n$ and $\ell\leq n$, 
satisfied with $\hat a_\ell^{(n)}\leqs 
{\rm e}^{-\nu^\prime (n-\ell )}$. Then
\be
\label{eq:conv-to-zero-exp} 
\lim_{n\to\infty} \sum_{\ell\leq n} Z^{(n)}_\ell = 0, 
\ee
$\bbP$-a.s.\ and in $\bbL_2$. }In particular, 
{ 
assume that the asymptotic bound 
\eqref{eq:abar} is satisfied for an array 
 $\bigl\{ b^{(n)}_{x,\ell} (y ,m )\bigr\}$  
with some $\nu >0$ and $\hat b_{x, \ell}^{(n)} \leqs 1$.
}
 Then
\be
\label{eq:largem}
\lim_{n\to\infty}
\sum_{\ell\leq n}
\sum_{x}\tq{x,\ell}\sum_{m >n-\ell}\sum_y b^{(n)}_{x,\ell} (y ,m )\lb
\fxq{y ,m } - \fa{y, m}\rb 
{ 
\df \lim_{n\to\infty} \sum_{\ell\leq n} Z^{(n)}_\ell
}
= 0 ,
\ee
$\bbP$-a.s.\ and in $\bbL_2$.
\begin{proof}
{
By~\eqref{eq:GUpper2},
\be 
\label{eq:conv-zero}
\bbE\bigl[\bigl( \bbE(Z^{(n)}_\ell \,|\, \calA_{\ell-k}) \bigr)^2\bigr],\;
{
\bbE\bigl[\bigl(
Z^{(n)}_\ell - \bbE(Z^{(n)}_\ell \,|\, \calA_{\ell+k})\bigr)^2\bigr]
}
\leqs \frac{{\rm e}^{-2\nu^\prime (n-\ell )}}{\ell^{d/2 -\epsilon} (1+k)^{1+\epsilon}}
\ee
{
Applying~\eqref{eq:Maximal} for each $n=1,2,\dots$ (with $\psi_\ell^2 = 
\bigl(\psi_\ell^{(n)}\bigr)^2 =
\frac{{\rm e}^{-2\nu^\prime (n-\ell )}}{\ell^{d/2 -\epsilon}}$),
we infer that
\[
\bbE\bigl[\bigl(\sum_{\ell\leq n} Z_\ell^{(n)}\bigr)^2\bigr] \leqs
\sum_{\ell = 1}^{n}
\frac{{\rm e}^{-2\nu^\prime  (n-\ell )}}{\ell^{d/2 -\epsilon}}. 
\]
}%
Since $d\geq 3$ and $\epsilon <{1/2}$, this implies that
$
\sum_{n}\bbE \bigl[\bigl( \sum_{\ell\leq n} Z_\ell^{(n)}\bigr)^2\bigr] < \infty 
.
$
\newline
Consider now the left-hand side of~\eqref{eq:largem}. For each $\ell\leq n$, 
the $Z^{(n)}_\ell$-sum on the right-hand side of~\eqref{eq:largem} can be 
rewritten in the form~\eqref{eq:Znl} with $a^{(n)}_{x,\ell} (y ,m) = 
b^{(n)}_{x,\ell} (y ,m)\One_{\{m>n-\ell\}}$. 
In this case, the inequality~\eqref{eq:abar} is satisfied 
for the array $\lbr a^{(n)}_{x,\ell} (y ,m)\rbr$ 
with any $\nu^\prime < \nu/2$ and  $\hat
a^{(n)}_{{x} , \ell} \leqs {\rm e}^{-\nu (n-\ell )/2}\df \hat a_\ell^{(n)}$.}
\end{proof}
\end{lemma}

\subsection{Multi-Dimensional Renewal and Asymptotics of $\tq{n}$}
Let us turn to the quenched asymptotics of $\tq{n}$. By construction,
\be
\label{eq:htqFormula}
\tq{z, n} = \sum_{m=0}^{n-1}\sum_x \tq{x ,m} f^{\theta_x\omega}_{z-x ,n-m}
\ \ {\rm and}\ \ \tq{ n} = \sum_z \tq{z, n} .
\ee
The claim~\eqref{eq:claima}  of Theorem~\ref{thm:B} follows from:
\begin{theorem}
\label{thm:tqnAsympt}
Assume that~\eqref{eq:WDBound} holds. Then,
\be
\label{eq:tqnAsympt}
\lim_{n\to\infty} \tq{n} = \frac{1}{\kappa}\Bigl( 1 + \sum_{x,y} \tq{x}\,
\bigl( f^{\theta_x\omega}_{y-x}  -\fa{y-x}\bigr) \Bigr\}
\df \frac{1}{\kappa} s^\omega \in (0, \infty ) ,
\ee
$\bbP$-a.s.\ and in $\bbL_2$ on the event $\lbr 0\in {\rm Cl}_\infty^{h } (V)
\rbr$.
\end{theorem}
\begin{proof}
Part of the proof appeared in Subsection~5.3 of
the review paper \cite{IV-Erwin}.
We rely on an expansion  similar
to the one employed by Sinai \cite{Sinai} and rewrite~\eqref{eq:htqFormula} as
(see {the beginning of Section~5.3} of \cite{IV-Erwin}   for details)
\be
\label{eq:Sinai0}
\tq{z ,n} = \ta{z, n} +
\sum_{\ell =0}^{n-1}\sum_{m=1}^{n-\ell}\sum_{r=0}^{n-\ell-m}
\sum_{x, y}
\tq{x ,\ell}\lb  f^{\theta_x\omega}_{y-x , m} - \fa{y-x ,m}\rb \ta{z-y ,r} .
\ee
In this way, $\tq{n}$ ({(56) in Section~5.3} of \cite{IV-Erwin}) can be 
represented as
\be
\label{eq:Sinai}
\tq{n} = \frac{1}{\kappa} s^\omega_n
+  \epsilon_n^\omega + \bigl( \ta{n} - \frac{1}{\kappa}\bigr)
\\
\ee
where
\be
\label{eq:snterm}
s^\omega_n = 1 + \sum_{\ell \leq n} \sum_x
\tq{x ,\ell}
 \lb\fxq{\, }  -1\rb,
\ee
and
the correction term $\epsilon_n^\omega =
- \epsilon_{n ,1}^\omega + \epsilon_{n ,2}^\omega
$ is given by
\be
\label{eq:EpsTerm}
\epsilon_n^\omega =
- \frac{1}{\kappa} \sumtwo{\ell\leq n}{m >  n -\ell}\sum_x
\tq{x ,\ell}\lb f^{\theta_x\omega}_{m} - \fa{m}\rb  +
 \sum_{\ell +m +r =n} \sum_{x }
\tq{x ,\ell}\lb f^{\theta_x\omega}_{m} - \fa{m}\rb \bigl( \ta{r} 
-\frac{1}{\kappa}\bigr)
.
\ee
By~\eqref{eq:tanAsympt}  $\ta{n} - \frac1{\kappa}$ tends to zero.
We claim that, $\bbP$-a.s.,
\be
\label{eq:EpsConv}
\lim_{n\to\infty} s^\omega_n = s^\omega \ \ {\rm and}\ \
\sum_n \bbE [( \epsilon^ \omega_n )^2] <\infty .
\ee

\paragraph{Convergence of $s^\omega_n$.}

Following the discussion in Subsection~4.5 of~\cite{IV-Crossing}, one readily
verifies that $s^\omega >0$ on the event  $\lbr 0\in {\rm Cl}_\infty^{h } (V)
\rbr$.
It remains to check~\eqref{eq:EpsConv}.
\smallskip

Let us rewrite $s_n^\omega$ as
\be
\label{eq:Xl}
 s_n^\omega -1 = \sum_{\ell\leq n}
 \sum_x \tq{x ,\ell}\lb\fxq{\, }  -1\rb \df \sum_{\ell = 0}^n Z_\ell .
\ee
The representation complies with~\eqref{eq:Znl} and~\eqref{eq:abar} with
$\hat{a}^{(n)}_{x, \ell}\leqs 1$ { 
and any positive $\nu <c_2$. 
}
Hence, by~\eqref{eq:GUpper2},
\be
\label{eq:Condl-m}
\bbE\bigl[\bigl( \bbE(Z_\ell \,|\, \calA_{\ell-k}) \bigr)^2\bigr],\;
{
\bbE\bigl[\bigl(
Z_\ell - \bbE(Z_\ell \,|\, \calA_{\ell+k}) \bigr)^2\bigr]
}
\leqs
\frac{1}{\ell^{ d  /2 -\epsilon}}\cdot \frac{1}{(1+k)^{1+\epsilon}} .
\ee
Since $d\geq 3$ and $\epsilon <{1/2}$,
{
Remark~\ref{re:conv} applies
}
 and 
 $\lim_{n\to\infty} s^\omega_n = 1
+\sum_0^\infty Z_\ell$ converges $\bbP$-a.s.\ and in $\bbL_2$ . \qed 

\paragraph{The $\epsilon_n^\omega$ term.}
Let us turn now to the correction term $\epsilon_n^\omega$
in~\eqref{eq:EpsTerm}.
The first summand to estimate is
\be
\label{eq:eps1n}
 \epsilon_{n, 1}^\omega =
\sum_{\ell\leq   n}
\sum_x
\tq{x ,\ell}
\sum_{m >n- \ell} \lb f^{\theta_x\omega}_{m} - \fa{m}\rb
\ee
It tends to zero by Lemma~\ref{lem:largem}.
The second summand is
\[
 \epsilon_{n,2}^\omega =  \sum_{\ell +m +r =n} \sum_{x }
\tq{x ,\ell}\lb f^{\theta_x\omega}_{m} - \fa{m}\rb \bigl( \ta{r} 
-\frac{1}{\kappa}\bigr)
\]
Since  $\ta{r} - 1/\kappa$ is exponentially decaying in $r$,
it is easy to see that~\eqref{eq:abar} still holds with 
{
$a^{(n)}_{x,\ell} \leqs \hat a^{(n)}_\ell \df {\rm e}^{-c_4 (n-\ell )}$,
for any positive $\nu <c_2$ and some $c_4 = c_4 (\beta ) >0$,
}
and {Lemma~\ref{lem:largem} applies}.
\end{proof}

\subsection{Quenched CLT}

To facilitate notation set  $\alpha_n = \alpha/\sqrt{n}$. For $r=1, 2, \dots$
define
\[
\qS{r}{\omega}  (\alpha )\df
\sum_{z} \tq{z ,r} {\rm e}^{i \alpha \cdot (z - rv )}
\]
We are studying  $\qS{n}{\omega} (\alpha_n )$.
The asymptotics of $\aS{n} (\alpha_n )= \bbE \qS{n}{\omega} (\alpha_n )$ is
given in~\eqref{eq:AnCLT}.  Using~\eqref{eq:Sinai0},
\be
\label{eq:S1}
\qS{n}{\omega}  (\alpha_n  ) = \aS{n} (\alpha_n  ) +
\sum_{\ell +m +r =n} \sum_{x, y , z }
\tq{x ,\ell}\lb f^{\alpha _x\omega}_{y-x, m} - \fa{y-x, m}\rb \ta{z- y, r}
{\rm e}^{i   (z - nv )\cdot\alpha_n  }.
\ee
Define
\be
\label{eq:GFunction}
 g^\omega_m (\alpha ) = \sum_{{y}}{\rm e}^{i (y -mv )\cdot \alpha }\lb \fq{y ,m
}- \fa{y, m}\rb .
\ee
Note that $g^\omega_m ( 0 ) = \fq{m} - \fa{m}$ 
{ 
and that $g^\omega_m (\alpha ) - g^\omega_m ( 0 ) = 
\sum_{{y}}\lb {\rm e}^{i (y -mv )\cdot \alpha } - 1\rb \lb \fq{y ,m
}- \fa{y, m}\rb$. 
}
We can rewrite~\eqref{eq:S1} as
\be
\label{eq:CLTBasic}
\qS{n}{\omega}  (\alpha_n  )  = \aS{n} (\alpha_n  )
+ \sum_{\ell +m +r =n}
\aS{r} (\alpha_n)
\sum_x
\tq{x ,\ell}
{\rm e}^{i (x-\ell v)\cdot \alpha_n} g_m^{\theta_x\omega}
\lb \alpha_n\rb .
\ee
Expanding terms in the products $\aS{r} (\alpha_n ){\rm e}^{i (x-\ell v)\cdot
\alpha_n} g_m^{\theta_x\omega}\lb
\alpha_n\rb $
as
\[
 \aS{r} (\alpha_n) = \aS{n}(\alpha_n ) + \lb \aS{r} (\alpha_n ) -
\aS{n}(\alpha_n ) \rb
\]
and, accordingly,
\[
 {\rm e}^{i (x-\ell v)\cdot \alpha_n} =
1 + \lb {\rm e}^{i (x-\ell v)\cdot\alpha_n}
-1\rb ,\ \
g_m^{ \omega}
\lb \alpha_n\rb
=
g_m^{ \omega}
\lb 0 \rb
 + \lb g_m^{ \omega}
\lb \alpha_n \rb  -
g_m^{ \omega}
\lb 0 \rb\rb ,
\]
we rewrite~\eqref{eq:CLTBasic} as:
\be
\label{eq:S2}
\begin{split}
\qS{n}{\omega}  (\alpha_n  ) &= \aS{n} (\alpha_n  )
\Bigl( 1+
 \sum_{\ell +m \leq n}
\sum_x
\tq{x ,\ell}
\lb \fxq{m } - \fa{m}\rb
\Bigr)
\\
&+
\aS{n} (\alpha_n )
\sum_{\ell +m \leq n}
\sum_x \tq{x ,\ell}
\lb g^{\theta_x \omega}_m  \lb \alpha_n\rb -
g^{\theta_x \omega}_m (0)\rb
\\
&+
\sum_{\ell +m +r =n} \lb \aS{r} (\alpha_n ) - \aS{n} (\alpha_n )\rb
\sum_x
\tq{x ,\ell}
\lb \fxq{m } - \fa{m}\rb\\
&+
\aS{n} (\alpha_n )
\sum_{\ell +m \leq n}
\sum_x \tq{x ,\ell}
\lb {\rm e}^{i(x -\ell v)\cdot \alpha_n} -1\rb
\lb \fxq{m } - \fa{m}\rb
\\
&+ \text{cross-terms}
\\
&\df
\aS{n} (\alpha_n)
\Bigl( 1+
 \sum_{\ell +m \leq n}
\sum_x
\tq{x ,\ell}
\lb \fxq{m } - \fa{m}\rb
\Bigr)
+\sum_{i=1}^3 \eta_{n, i}^\omega +
 \text{cross-terms} .
\end{split}
\ee
By Theorem~\ref{thm:tqnAsympt} {the sequence of random factors of}
  $\aS{n} (\alpha)$
tend  to $s^\omega$.
The cross terms are of lower order and we shall {briefly discuss them at
the end of the present section}. The crux of the
matter is to prove:
\begin{theorem}
 \label{thm:QuCLT}
For every $\alpha\in\bbR^d$ the correction terms $\eta_{n,i}^\omega$
in~\eqref{eq:S2}
satisfy :
\be
\label{eq:CortermsCLT}
\text{For $i=1,2,3$}\ \ \lim_{n\to\infty}  \eta_{n,i}^\omega  = 0\ \
\text{$\bbP$-a.s.\ and
in $\bbL_2 (\Omega )$} .
\ee
\end{theorem}
Once~\eqref{eq:CortermsCLT} is established, we readily infer
from~\eqref{eq:tanAsympt}, \eqref{eq:AnCLT} and~\eqref{eq:claima} that
\be
\label{eq:QuCLT}
\lim_{n\to\infty} \frac{\qS{n}{\omega}  (\alpha/\sqrt{n} )}{\tq{n}} = 
\exp\bigl\{ -
\tfrac12\Sigma\, \alpha \cdot\alpha\bigr\} ,
\ee
$\bbP$-a.s.\ on the event
$\lbr 0\in {\rm Cl}_\infty^{h } (V) \rbr$
for every $\alpha\in\bbR^{d+1}$ fixed. This is precisely~\eqref{eq:claimc}
of Theorem~\ref{thm:B}.

\section{Correction Terms}
\label{sec:proofeta}
In this Section, we prove~\eqref{eq:CortermsCLT}. The  correction terms
$ \eta_{n,i}^\omega ;\, i=1, 2, 3,$ will be treated separately. 
Recall that we are working with $\epsilon <1/6$ such that~\eqref{eq:WDBound}
holds with 
$\rho = \epsilon/2$.

\paragraph{The $\eta_{n, 1}^\omega $ term .}
Consider
\[
\frac{\eta_{n ,1}^\omega}{\aS{n} (\alpha_n )} =
\sum_{\ell\leq n}\sum_x \tq{x,\ell}\sum_{m\leq n-\ell}\sum_y
\lb {\rm e}^{i (y-mv )\cdot\alpha_n} -1\rb\lb
\fxq{y,m} - \fa{y,m}\rb
\]
By Lemma~\ref{lem:largem}, the constraint $m\leq n-\ell$ might be removed, and
we need to prove the convergence to zero of
\be 
\label{eq:eta1-form}
 \hat\eta_{n, 1}^\omega \df
\sum_{\ell\leq n}\sum_x \tq{x,\ell}\sum_{m ,y}
a^{(n)}_{x, \ell} (y, m )
\lb
\fxq{y,m} - \fa{y,m}\rb \df  \sum_{\ell\leq n}   Z^{(n)}_\ell .
\ee
with $a^{(n)}_{x, \ell} (y, m ) =
\lb {\rm e}^{i (y-mv )\cdot\alpha_n} -1\rb$. 
\begin{lemma}
\label{lem:2Term}
In the very weak disorder regime,
\be
\label{eq:2Term}
\lim_{n\to\infty} \hat \eta_{n ,1}^\omega = 0 ,
\ee
$\bbP$-a.s.\ and in $\bbL_2$  for each $\alpha\in \bbR^d$ fixed.
\end{lemma}
\begin{proof}[Proof of Lemma~\ref{lem:2Term}]
For $a^{(n)}_{x, \ell} (y, m )$ as above, \eqref{eq:abar} is 
satisfied with 
{
$\hat a^{(n)}_{{x} ,\ell} \leqs \hat a^{(n)}_\ell \df  1/\sqrt{n}$
and any $\nu < c_2$.
}
 By
\eqref{eq:GUpper2} of Lemma~\ref{lem:LemUB-2}, 
\be
\label{eq:eta2-1}
\bbE\bigl[\bigl( \bbE(Z^{(n)}_\ell \,|\, \calA_{\ell-k}) \bigr)^2\bigr],\;
{
\bbE\bigl[\bigl( Z^{(n)}_\ell - \bbE(Z^{(n)}_\ell \,|\, \calA_{\ell+k}) 
\bigr)^2\bigr]
}
 \leqs
\frac{1}{n \ell^{d/2 -\epsilon}}\cdot\frac{1}{ (1+k)^{1+\epsilon}} .
\ee
By~\eqref{eq:Maximal},
$\Var \bigl( \hat \eta_{n, 1}^\omega\bigr) \leqs 1/n$. Consequently,
the lacunary sequence $\bigl\{ \hat\eta_{n^{1+\delta}, 1}^\omega \bigr\}$
converges to zero $\bbP$-a.s.\ and in $\bbL_2$ {for any $\delta >0$}.

It remains to {choose $\delta >0$ appropriately and to}
control fluctuations of $\hat\eta_{\cdot,1}^\omega$ on the intervals of the
form $[N, \dots , N+R]$ with
\be
\label{eq:nu2Range}
N\cong n^{1+\delta } \quad\text{and}\quad
R\cong (1+n)^{1+\delta} - n^{1+\delta}\cong n^\delta .
\ee
Now,
\be
\label{eq:eta2-2terms}
\hat \eta_{N +r , 1}^\omega - \hat\eta_{N ,1}^\omega =
\sum_{\ell\leq N} \bigl( Z^{(N+r)}_\ell - Z^{(N)}_\ell \bigr) +
\sum_{\ell=N+1}^{N+r} Z^{(N+r)}_\ell  .
\ee
We should not worry about the second term above:
\eqref{eq:Maximal} can still be applied to bound $\Var \bigl(
\sum_{\ell=N+1}^{N+r} Z^{(N+r)}_\ell \bigr)$ for each $r$ fixed.
By~\eqref{eq:eta2-1} and the union bound,
\[
\bbE\Bigl[
\max_{r\leq R}
\bigl(
\sum_{\ell = N+1}^{N+r}Z^{(N+r)}_\ell   \bigr)^2 \Bigr]
\leq \sum_{r=1}^R
\bbE
\Bigl[\bigl(
\sum_{\ell = N+1}^{N+r}Z^{(N+r)}_\ell \bigr)^2\Bigr]
\leqs \frac{R}{N^{d/2 - \epsilon }}
\cong \frac{1}{n^{{(1 +\delta) \lb \frac{d}{2} -\epsilon -\frac{\delta}{1+\delta}\rb}}} .
\]
The right-hand side above is summable (in $n$) by our
choice~\eqref{eq:nu2Range} {whenever $\frac{d}{2} -\epsilon -\frac{\delta}{1+\delta} >1$.
Since $d\geq 3$ and $\epsilon <1/2$, there are feasible choices of $\delta >0$ to ensure the latter.}
\newline
As for the first term in~\eqref{eq:eta2-2terms}, note that for $\ell\leq N$,
\be
\label{eq:eta2-coef}
a^{(N+r)}_{x,\ell} (y ,m ) -   a^{(N)}_{x,\ell} (y ,m )
=
\lb {\rm e}^{i (y - mv )\cdot\alpha_{N+r}}
- {\rm e}^{i (y - mv )\cdot\alpha_{N}}\rb
\df
b^{(N, r)}_{x,\ell} (y ,m ) .
\ee
The array $\bigl\{ b^{(N, r)}_{x,\ell} (y,m) \bigr\}$ 
satisfies~\eqref{eq:abar} with 
{
$\hat b^{(N, r)}_{x, \ell} \leqs \hat b^{(N, r)}_{\ell} \df  r/N^{3/2}$ 
 and any $\nu < c_2$.}
By~\eqref{eq:GUpper2}, \eqref{eq:Maximal} and the union bound,
\[
\bbE\Bigl[\max_{r\leq R} \Bigl( \sum_{\ell\leq N} \bigl(Z^{(N+r)}_\ell -
Z^{(N)}_\ell\bigr)\Bigr)^2\,\Bigr]
\leqs  \frac{R^3}{N^3} .
\]
By our choice~\eqref{eq:nu2Range}, $\frac{R^3}{N^3} \eqvs n^{-3}$ for any 
choice of $\delta >0$, and, consequently, the right-hand side above is summable.
\end{proof}

\paragraph{The $\eta_{n,2}^\omega $ term.} By~\eqref{eq:tanAsympt},
\be
\label{eq:E1}
 \frac{\aS{r}(\alpha_n)}{\aS{n}(\alpha_n )}
{= \frac{\ta{r} (i\alpha_n ){\rm e}^{-ir v\cdot\alpha_n}}
{\ta{n} (i\alpha_n ){\rm e}^{-in v\cdot\alpha_n}}}
= {\rm e}^{(r-n)\lb \mu (i\alpha_n )  -i v\cdot\alpha_n\rb}
\lb 1 + o({\rm e}^{-c_4 r})
\rb .
\ee
{ 
Set $\phi (\alpha ) = iv\cdot\alpha - \mu (i\alpha )$. The function $\phi$ is
defined
in a neighbourhood of the origin
and it is of quadratic growth  there. 
By Lemma~\ref{lem:largem}, the residual term $\smo{{\rm e}^{-c_4 r}}$ is 
negligible. 
Next, for $\ell\leq n$
the coefficients 
\be 
\label{eq:alm-case2}
a^{(n)}_{x, \ell }(y ,m ) = {\rm e}^{(m+\ell) \phi (\alpha_n )} - 1 
\ee
satisfy~\eqref{eq:abar} with $\hat a^{(n)}_{{x}, \ell} \leqs \hat a^{(n)}_{\ell}\df \ell/ n$
{and any $\nu < c_2$}.
}
Consequently, \eqref{eq:largem} enables to lift the restriction $m\leq n-\ell$.
  Therefore, we need to prove
convergence to zero of
\be
\label{eq:eta2prime}
\hat\eta_{n ,2}^\omega
=
\sum_{\ell\leq n}\sum_x \tq{x, \ell}\sum_m \lb {\rm e}^{(m+\ell) \phi (\alpha_n )} - 1\rb
\lb \fxq{m} - \fa{m}\rb  \df \sum_{\ell\leq n}Z^{(n)}_\ell .
\ee
\begin{lemma}
\label{lem:1Term}
In the very weak disorder regime
\be
\label{eq:1Term}
\lim_{n\to\infty} \hat \eta_{n ,2}^\omega = 0 ,
\ee
$\bbP$-a.s.\ and in $\bbL_2$  for each $\alpha\in \bbR^d$ fixed.
\end{lemma}
\begin{proof}[Proof of Lemma~\ref{lem:1Term}]
Recall that for in $a^{(n)}_{x, \ell }(y ,m )$ defined in \eqref{eq:alm-case2} the 
asymptotic bound \eqref{eq:abar} is satisfied with $\hat a^{(n)}_{{x}, \ell} \leqs  \ell/ n$
uniformly in $n$, $\ell\leq n$ and $x$.  
By~\eqref{eq:GUpper2} and~\eqref{eq:Maximal},
\be
\label{eq:VarEn1}
 \Var \lb\hat \eta_{n ,2}^\omega\rb  \leqs \frac1{n^2} \sum_{\ell\leq n}
\frac{1}{\ell^{d/2 -2 -\epsilon}} \eqvs \frac{1}{n^{\frac{d}{2} - 1 -\epsilon} 
\wedge n^2}
,
\ee
which  already implies  the claim of Lemma~\ref{lem:1Term} in dimensions
$d\geq 5$. We shall continue discussion for the most difficult case of $d=3$.
{ \eqref{eq:VarEn1} implies that
} 
\be
\label{eq:En1Sub}
\bbE \bigl[\sum_{n}  \lb \hat \eta_{n^{2+\delta} ,2}^\omega \rb^2\bigr] < 
\infty\Rightarrow
\lim_{n\to\infty}\hat \eta_{n^{2+\delta} ,2}^\omega  = 0\ \ \bbP-{\rm a.s.}\
{\rm and\  in}\ \bbL_2 , 
\ee
{ 
whenever
\be
\label{eq:delta-eps}
(2+\delta)\lb \tfrac{1}{2}-\epsilon\rb >1 \quad\text{, that is,}\quad 
\delta>\frac{4\epsilon}{1-2\epsilon}.
\ee
Since $\epsilon < 1/6$, there are  choices of $\delta\in (0,1)$ which comply with 
\eqref{eq:delta-eps}. }
We need to control fluctuations of $\hat \eta_{N+r ,2}^\omega - \hat \eta_{N
,2}^\omega$ on the intervals of the form $[N, \dots ,N+R]$, where
\be
\label{eq:Intervals}
N \cong n^{2+\delta}\ \ {\rm and}\ \ R \cong (n+1)^{2+\delta} - n^{2+\delta}
\cong n^{1+\delta} .
\ee
Consider the following decomposition:
\be
\label{eq:ENk}
\begin{split}
\hat\eta_{N+r ,2}^\omega - \hat\eta_{N ,2}^\omega &=
{
\sum_{\ell \leq {N }} \sum_{x}\tq{x,\ell} \sum_m \lb {\rm e}^{(m+\ell )\phi (\alpha_{N+r})}
- {\rm e}^{(m+\ell )\phi (\alpha_{N})}\rb\lb \fxq{m } - \fa{m}\rb 
+\sum_{\ell = N+1}^{N+r} Z_\ell^{(N+r )}}\\
& =
\sum_{\ell \leq {N+r }} \sum_{x}\tq{x,\ell} \sum_m \lb {\rm e}^{(m+\ell )\phi (\alpha_{N+r})}
- {\rm e}^{(m+\ell )\phi (\alpha_{N})}\rb\lb \fxq{m } - \fa{m}\rb 
+\sum_{\ell = N+1}^{N+r} Z_\ell^{(N)} .
\end{split}
\ee
{The terms $Z_\ell^{(N)}$ in the second sum above were defined in \eqref{eq:eta2prime} and they 
 do not depend on $r$. 
By~\eqref{eq:Maximal}}, we are entitled to control its maximum on the
interval $[N, \dots , N+R]$:
\be
\label{eq:MaxEn1}
\begin{split}
\bbE \Bigl[ \max_{r\leq R} \bigl(
\sum_{\ell = N+1}^{N+r} Z_\ell^{(N)} \bigr)^2 \Bigr]
&\leqs  \frac1{N^2} \sum_{\ell =N+1}^{N+R}\frac{\ell^2}{\ell^{3/2 -\epsilon}}
\cong  \frac1{N^2}\bigl\{ (N+R)^{3/2 +\epsilon} - N^{3/2 +\epsilon}\bigr\}\\
&\cong \frac{R}{N^{3/2 - \epsilon}}
\cong
\frac{n^{1+\delta}}{n^{3(1+\delta/2 ) -\epsilon (2+\delta )}} 
\cong 
\frac{1}{n^{2 +\frac{\delta}{2} - \epsilon (2+\delta )}} 
\df a_n,
\end{split}
\ee
by our choice of parameters~\eqref{eq:Intervals}.

For each $r\leq R$ the first term in~\eqref{eq:ENk} corresponds to the following choice of
coefficients in the representation~\eqref{eq:Znl}: 
$a^{(N+r)}_{x , \ell} (y , m ) = \lb {\rm e}^{(m+\ell )\phi (\alpha_{N+r})}
- {\rm e}^{(m+\ell )\phi (\alpha_{N})}\rb $. Thus, \eqref{eq:abar} is satisfied
with ${\hat a^{(N+r)}_{{x},\ell}\leqs \hat a^{(N+r)}_{\ell} \df \ell r/N^2 }$ 
{and any $\nu <c_2$}. 
By the very same~\eqref{eq:GUpper2} and~\eqref{eq:Maximal}, we infer that, for
any $r\leq R$,
\be
\label{eq:Var1d1}
\Var \Bigl(
\sum_{\ell = 1}^{{N +r}}
\sum_{x}\tq{x,\ell} \sum_m \lb {\rm e}^{(m+\ell )\phi (\alpha_{N+r})}
- {\rm e}^{(m+\ell )\phi (\alpha_{N})}\rb \Bigr)
{ 
\leqs \frac{r^2}{N^4}\sum_{\ell = 1}^{{N +r}} \ell^{\frac{1}{2} +\epsilon}
}
\leqs \frac{R^2}{N^{5/2 -\epsilon}}.
\ee
Hence, by the union bound and our choice of parameters~\eqref{eq:Intervals},
\be
\label{eq:MaxEn12}
\begin{split}
\bbE \Bigl[ \max_{r\leq R} 
&\Bigl(
\sum_{\ell =1}^{{N+r}}
\sum_{x}\tq{x,\ell} \sum_m \bigl( {\rm e}^{(m+\ell )\phi (\alpha_{{N+r}})}
- {\rm e}^{(m+\ell )\phi (\alpha_{N})}\bigr) \Bigr)^2\,\Bigr]
\\
&\qquad \leqs
\frac{R^3}{N^{5/2 -\epsilon} } \cong
\frac{1}{n^{2 -\frac{\delta}{2} -\epsilon (2+\delta )}} 
\df b_n .
\end{split}
\ee
{Since $\epsilon < 1/6$, the inequality $\frac{\delta}{2} + (2+\delta )\epsilon <1$ 
holds for any choice of 
$\delta \leq 1$. } Therefore, any such choice ensures that $\sum_n (a_n +b_n )
<\infty$, which implies that 
\[
 \lim_{n\to\infty} \max_{n^{2+\delta}\leq r < (n+1)^{2+\delta}}
\abs{\hat\eta_{r, 1}^\omega - \hat\eta_{n^{2+\delta}, 1}^\omega } = 0 ,
\]
$\bbP$-a.s and in $\bbL_2$. 
The proof of Lemma~\ref{lem:1Term}  is completed.
\end{proof}

\paragraph{The $\eta_{n,3}^\omega$ term.}
{This is the most difficult term, and, at this stage,  we need to
 rely on Lemma~\ref{lem:LemUB-1}
rather than on Lemma~\ref{lem:LemUB-2}.}
Recall that
\[
 \frac{\eta_{n,3}^\omega}{\aS{n} (\alpha_n ) }
=\sum_{\ell\leq n}\sum_x \tq{x, \ell}\lb {\rm e}^{i (x - \ell v )\cdot\alpha_n} - 1\rb
\sum_{m\leq n-\ell} \lb \fxq{m} - \fa{m}\rb .
\]
By Lemma~\ref{lem:largem}, we may remove the constraint $m\leq n-\ell$.
Define, therefore,
\be
\label{eq:Znl3}
 Z^{(n)}_\ell = \sum_x \tq{x, \ell}\lb {\rm e}^{i (x - \ell v )\cdot\alpha_n} - 1\rb
\lb \fxq{\, } - 1\rb\quad
\text{and}\quad
\hat\eta_{n ,3}^\omega = \sum_{\ell\leq n}Z^{(n)}_\ell .
\ee
We need to prove:
\begin{lemma}
\label{lem:3Term} In the very weak disorder regime,
\be
\label{eq:3Term}
\lim_{n\to\infty} \hat\eta_{n, 3}^\omega  = 0 ,
\ee
$\bbP$-a.s.\ and in $\bbL_2$  for each $\alpha\in \bbR^d$ fixed.
\end{lemma}
\begin{proof}[Proof of Lemma~\ref{lem:3Term}]
For $Z^{(n)}_\ell$ defined in~\eqref{eq:Znl3}, the bound~\eqref{eq:abar} is
satisfied with 
\be 
\label{eq:an-for3}
\hat a^{(n)}_{x , \ell} = \abs{{\rm e}^{i (x - \ell v )\cdot\alpha_n} - 1}
\cong 
{
\frac{\abs{x-\ell 
v}}{\sqrt{n}}
}
\wedge 1,
\ee 
{for any $\nu <c_2$}. 
Applying~\eqref{eq:GUpper1}, we infer that
\be
\label{eq:Condl-mY1}
 \bbE\bigl[\bigl( \bbE(Z_\ell^{(n)} \,|\, \calA_{\ell-k}) \bigr)^2\,\bigr]
\leqs 
\frac1{\ell^{d+1 - \epsilon/2}}
\sum_{x\in\cHm{\ell -k}} {\rm e}^{- c_2 \abs{x - \ell v}^2/\ell}\,
\bigl(
\frac{\abs{x-\ell v}^2}{n}\wedge 1\bigr) .
\ee
As in the derivation of~\eqref{eq:GUpper2}, we may assume that $k$ is
sufficiently large, so that, in particular, $\abs{x-\ell v}\geq 
\frac{k\abs{v}}{2}$ for all $x\in \cHm{\ell -k}$.  In the latter case, the sum 
on the right-hand side
of~\eqref{eq:Condl-mY1} is bounded above by
\be
\label{eq:Condl-mY}
\begin{split}
&\leqs  
\int_{\abs{y} >\frac{k\abs{v}}{2}} {\rm e}^{-c_2 \abs{y}^2/\ell}\,
\Bigl(\frac{\abs{y}^2}{n}\wedge 1\Bigr) \dd y
= 
\int_{\frac{k\abs{v}}{2}}^\infty r^d {\rm e}^{-c_2 r^2/\ell}\,
\Bigl( \frac{r^2}{n}\wedge 1\Bigr) \dd r \\
&\cong
{\ell^{(d+1)/2}}
 \int_{\frac{k\abs{v}}{2\sqrt{\ell}}} ^\infty
t^d {\rm e}^{-c_2 t^2}\, \Bigl( \frac{t^2\ell}{n}\wedge 1\Bigr) \dd t
 \df
{\ell^{(d+1)/2}} {I_n (\ell ,k)} .
\end{split}
\ee
We shall repeatedly rely on~\eqref{eq:m-lbound}. There are two cases to
consider:

\noindent
\case{1}  If $\frac{k\abs{v}}{2}\leq \sqrt{n}$, then
\be
\label{eq:En3Case1}
\begin{split}
I_n (\ell ,k)  &= \frac{\ell}{n} \int_{\frac{k\abs{v}}{2\sqrt{\ell}}}^{\sqrt{n/\ell}}
t^{d+2} {\rm e}^{-c_2 t^2}\dd t  + \int_{\sqrt{n/\ell}}^\infty  t^d {\rm e}^{-c_2 t^2}\dd t
\leqs \frac{\ell}{n} {\rm e}^{-c_2
\frac{ (k{\abs{v}})^2}{ 8 \ell}} + {\rm e}^{-c_2 \frac{n}{2\ell}} \\
&
\leqs   \frac{\ell}{n}{\rm e}^{-c_2
\frac{ (k{\abs{v}})^2}{ 8 \ell}} \leqs 
\frac{1}{n}\cdot  \frac{\ell^{3/2 +\epsilon/2}}{(1+k
)^{1+\epsilon}}.
\end{split}
\ee

\noindent
\case{2} If $\frac{k\abs{v}}{2} > \sqrt{n}$, then
\be
\label{eq:En3Case2}
I_n (\ell ,k)  =
\int_{\frac{k\abs{v}}{2\sqrt{\ell}}}^\infty  t^d {\rm e}^{-c_2 t^2}\dd t \leqs
{\rm e}^{-c_2 \frac{(k{\abs{v}})^2}{8\ell}}
\leqs \frac{\ell}{n}{\rm e}^{-c_2 \frac{(k{\abs{v}})^2}{16\ell}}
\leqs  \frac{1}{n}\cdot \frac{\ell^{3/2 +\epsilon/2}}{(1+k )^{1+\epsilon}}
\ee
as well.

\noindent
As a result (again we restrict attention to the most difficult case $d=3$):
\be
\label{eq:EYlnBound}
\bbE\bigl[\bigl( \bbE(Z_\ell^{(n)} \,|\, \calA_{\ell-k}) \bigr)^2 \bigr] \leqs
\frac{1}{(1+k )^{1+\epsilon}} \cdot \frac{1}{n}\cdot\frac{1}{ \ell^{1/2 
-\epsilon }}.
\ee
{ 
Let us turn to the bound~\eqref{eq:GUpper1-plus}
on $\bbE\bigl[\bigl( Z_\ell^{(n)} - \bbE(Z_\ell^{(n)} \,|\, 
\calA_{\ell+k})\bigr)^2\bigr]$.
As before, we apply it with $\hat a^{(n)}_{x,\ell} = \frac{\abs{x-\ell 
v}}{\sqrt{n}} \wedge 1$.
We need to estimate
\be 
\frac1{\ell^{d+1 - \epsilon/2}}
\sum_{x\in\cHm{\ell +k}} {\rm e}^{- c_2 \abs{x - \ell v}^2/\ell - \nu {\rm d}_{\ell +k }(x)}
\bigl(
\frac{\abs{x-\ell v}^2}{n}\wedge 1\bigr).
\ee
Proceeding as in \eqref{eq:extra-term} and noting that $\frac{k^2}{n}{\rm 
e}^{-\nu^\prime k} \leqs 
\frac{1}{n}{\rm e}^{-\nu^\prime k/2}$, we infer that the upper 
bound~\eqref{eq:EYlnBound} holds for  
$\bbE\bigl[\bigl( Z_\ell^{(n)} - \bbE(Z_\ell^{(n)} \,|\, 
\calA_{\ell+k})\bigr)^2\bigr]$ as well. \newline
}%
In particular, by~\eqref{eq:Maximal},  $\Var \lb \hat\eta^\omega_{n, 3}\rb \leqs
n^{-1/2 +\epsilon}$
and, as in the case
of $\hat\eta^\omega_{n , 2}$, we infer that there is $\bbP$-a.s.\
and $\bbL_2$ convergence to zero along lacunary
subsequences  $\bigl\{ n^{2+\delta } \bigr\}$, {whenever
$\delta$ satisfies \eqref{eq:delta-eps}}.
Hence, again as in the case of $\hat\eta^\omega_{n, 2}$,  we need
to control the fluctuations $\hat\eta^\omega_{N+r, 3} - \hat\eta^\omega_{N, 3}$ 
over intervals of the form~\eqref{eq:Intervals}. {As in~\eqref{eq:ENk}}, 
we make use of the decomposition
\be 
\label{eq:eta3-decomp}  
\begin{split}
\hat\eta^\omega_{N+r, 3} - \hat\eta^\omega_{N, 3}
  &=
\sum_{\ell =1}^{{N+r}}
\sum_x \tq{x ,\ell}
\lb {\rm e}^{i (x -\ell v)\cdot \alpha_{N+r}} -
{\rm e}^{i (x -\ell v)\cdot\alpha_{N}}
\rb
\lb \fxq{\, } - 1 \rb
+
\sum_{\ell = N+1}^{N+{r}} Z_\ell^{(N)} \\ 
& \df \sum_{\ell =1}^{{N+r}} Z_\ell^{(N,r)} + \sum_{\ell = N+1}^{N+{r}} Z_\ell^{(N)}
\end{split}
\ee
{
We continue to work with $d=3$. For each $r=1, \dots , R$ fixed the 
bound \eqref{eq:abar} is satisfied with 
\be 
\label{eq:an-for3-diff}
\abs{ {\rm e}^{i (x -\ell v)\cdot \alpha_{N+r}} -
{\rm e}^{i (x -\ell v)\cdot\alpha_{N}}}  \leqs \frac{R\abs{x -\ell v}}{N^{3/2}}\wedge 1
\df \hat a^{(N)}_{x, \ell} .
\ee
The expression  for  $\hat a^{(N)}_{x, \ell}$ in \eqref{eq:an-for3-diff} is similar to 
\eqref{eq:an-for3} with $1/\sqrt{n}$ being replaced by the higher order term 
$R/N^{3/2}$.  Literally repeating the derivation of \eqref{eq:EYlnBound} we infer 
that for each $r=1, \dots , R$ fixed random variables $Z_\ell^{(N,r)}$ in 
\eqref{eq:eta3-decomp} satisfy:
\be
\label{eq:EYlnBound-R}
\bbE\bigl[\bigl( \bbE(Z_\ell^{(N,r)} \,|\, \calA_{\ell-k}) \bigr)^2 \bigr] , 
\bbE\bigl[\bigl(Z_\ell^{(N,r)} - \bbE(Z_\ell^{(N,r)} \,|\, \calA_{\ell+k}) \bigr)^2 \bigr] 
\leqs
\frac{1}{(1+k )^{1+\epsilon}} \cdot \frac{R^2}{N^3}\cdot\frac{1}{ \ell^{1/2 
-\epsilon }}. 
\ee
Applying \eqref{eq:Maximal} we conclude:}
For $N$ and $R$ in the range~\eqref{eq:Intervals}, and for any $r=1, \dots, R$ fixed, 
 the variance of the first
term on the right hand side of \eqref{eq:eta3-decomp} is  uniformly bounded above ($d=3$) by
\[
\leqs \frac{R^2}{N^{3}} \cdot N^{1/2 +\epsilon}  =
\frac{R^2}{N^{5/2 -\epsilon}} .
\]
As in the case of~\eqref{eq:MaxEn12}, the union bound suffices. 

Finally, using~\eqref{eq:EYlnBound} as an input for~\eqref{eq:Maximal}, we 
conclude that
\[
\bbE\bigl[\bigl(\max_{r\leq R} \sum_{\ell = N+1}^{N+{r}} 
Z_\ell^{(N)}\bigr)^2\bigr]
\leqs
\frac1{N}\sum_{\ell = N+1}^{N+R}\frac{1}{\ell^{1/2 -\epsilon}}
\cong
\frac{\lb N+R\rb^{1/2 +\epsilon} - N^{1/2 +\epsilon}}{N}
\cong \frac{R}{N^{3/2 -\epsilon}} .
\]
{
The right-most term above is
summable in $N$ for any choice of $\delta\in (0,1)$ in \eqref{eq:Intervals}. 
In particular, it is summable if $\delta$ complies with \eqref{eq:delta-eps}.
}%
Consequently, for such choices of $\delta$,
\[
\lim_{n\to\infty} \max_{n^{2+\delta}\leq r < (n+1)^{2+\delta}}
\abs{\hat\eta_{r , 3}^\omega - \hat\eta_{n^{2+\delta}, 3}^\omega} = 0 ,
\]
$\bbP$-a.s.\ and in $\bbL_2$,  and we are home.
\end{proof}
\paragraph{{Cross-terms in~\eqref{eq:S2}.}} 
{
Since our treatment of the correction terms
$\eta_{n,i}^\omega$ was based either on the estimates~\eqref{eq:abar} on the 
absolute value of the coefficients $a_{x , \ell}^{(n)} (y ,m )$ which appear 
in~\eqref{eq:eta1-form}, \eqref{eq:eta2prime} and~\eqref{eq:Znl3}, or on 
estimates on the absolute value of differences 
$a_{x , \ell}^{(N+r)} (y ,m ) - a_{x , \ell}^{(N)} (y ,m )$ between these coefficients, which show 
up in the decompositions~\eqref{eq:eta2-2terms}, \eqref{eq:ENk} and, 
respectively, \eqref{eq:eta3-decomp}, we readily infer 
that~\eqref{eq:CortermsCLT} of Theorem~\ref{thm:QuCLT} carry over to cross-terms
in~\eqref{eq:S2}. For instance, let us consider the (2-3) cross term
\[
\aS{n} (\alpha_n  )
\sum_{\ell\leq n}\sum_{x, m}\tq{x,\ell} \lb {\rm e}^{(m+\ell )\phi (\alpha_n )} 
- 1\rb\lb 
{\rm e}^{i (x- \ell v)\cdot\alpha_n } -1\rb \lb \fxq{m } - \fa{m}\rb \df 
\sum_{\ell\leq n} Z^{(n)}_\ell.
\]
In terms of \eqref{eq:Znl}, we are working with coefficients
\[
a^{(n)}_{x , \ell} (y ,m ) =a^{(n)}_{x , \ell} (m ) =
\lb {\rm e}^{(m+\ell)\phi (\alpha_n )} - 1\rb\lb 
{\rm e}^{i (x- \ell v)\cdot\alpha_n } -1\rb .
\]
In particular, \eqref{eq:VarEn1} carries over, and we infer convergence to zero 
along the lacunary sequence $n^{2+\delta}$. In order to study the fluctuations 
of $\sum_{\ell\leq n}Z^{(n)}_\ell$ on the intervals $[N, \dots , N+R]$, one, as 
was done in~\eqref{eq:ENk}, employs the decomposition
\be 
\label{eq:cross2-3}
\begin{split}
\sum_1^{N+r} &Z^{(N+r )}_\ell - \sum_1^{N} Z^{(N )}_\ell \\
&\quad  = 
\sum_{\ell \leq {N+r }} \sum_{x,m}\tq{x,\ell} 
\bigl( 
a^{(N+r)}_{x , \ell} (m ) - a^{(N)}_{x , \ell} (m )
\bigr)
\lb \fxq{m} - \fa{m}
\rb 
+\sum_{\ell = N+1}^{N+r} Z_\ell^{(N )}
\end{split}
\ee
Since~\eqref{eq:VarEn1} holds,  
the second term on the right-hand side of \eqref{eq:cross2-3} is worked out 
exactly as in~\eqref{eq:MaxEn12}. On the other hand,
\be
\label{eq:a-23-split}
\begin{split}
 a^{(N+r)}_{x , \ell} (m ) - a^{(N)}_{x , \ell} (m ) &= 
\lb {\rm e}^{(m+\ell )\phi (\alpha_{N+r} )} - {\rm e}^{(m+\ell )\phi (\alpha_{N} )}\rb
\lb 
{\rm e}^{i (x- \ell v)\cdot\alpha_{N+r }} -1\rb \\
&\quad + 
\lb 1-{\rm e}^{(m+\ell)\phi (\alpha_{N})}\rb
\lb{\rm e}^{i(x-\ell v)\cdot\alpha_{N}}-
{\rm e}^{i (x- \ell v)\cdot\alpha_{N+r }}\rb
\end{split}
\end{equation}
In view of Remark~\ref{rem:tn}, we restrict attention to $\ell\leq N+R$, which 
implies that
$\abs{\ell \phi (\alpha_{N} )}\leqs 1$ uniformly in all the situations in question. Hence both terms
above can be worked out exactly as in the cases of, respectively, the 
corrections $\eta_{n,2}^\omega$ and $\eta_{n, 3}^\omega$. 
}
\section{Proof of the $\bbL_2$ estimate~\eqref{eq:WDBound}}
\label{sec:WDBound}

A variant of our target estimate~\eqref{eq:WDBound} was proved
in~\cite[Proposition~3.1]{IV-Crossing} 
and we shall follow a similar line of reasoning
and, eventually, rely on upper bounds derived in the latter paper.

\subsection{Preliminaries}
For $u,v\in\Zdo$ and $m\in\bbN$, we set
\[
\tq{u,v,m} \df t^{\theta_u\omega}_{v-u,m},\quad \fq{u,v,m} \df f^{\theta_u\omega}_{v-u,m},\quad \ta{u,v,m}
\df \bbE(\tq{u,v,m}),\quad \fa{u,v,m} \df \bbE(\fq{u,v,m}),
\]
and
\[
D(u,v) \df ( u +\calY^h ) \cap ( v - \calY^h ) \, \cap \, \bbZ^{d+1} . 
\]
Moreover, we write $\calT(u,v;n)$ for the set of all cone-confined
paths $\gamma\subseteq D(u,v)$ of length $n$ leading from $u$ to $v$, and $\calF
(u,v;n)$ for the corresponding subset of irreducible paths.

\medskip
Observe first that, by definition, $\fq{u,v,m}$ is
$\sigma\setof{V(x)}{x\in D(u,v)}$-measurable. In particular, if
$D(x,y)\cap D(\xpr,\ypr) = \emptyset$, then
\begin{equation}
\label{eq:L2factorize}
\bbE \bigl[ \tq{x,\ell}\, \tq{\xpr,\ell}\, \bbE (\fq{x,y,m}-\fa{x,y,m}\given\calA)\, \bbE
(\fq{\xpr,\ypr,\mpr}-\fa{\xpr,\ypr,\mpr}\given\calA) \bigr] = 0.
\end{equation}
Indeed, in that case, either $D(x,y)\cap (\xpr-\calY^h) =\emptyset$, or $D(\xpr,\ypr)\cap (x-\calY^h) =\emptyset$. For
definiteness, let us assume the latter. We can then conclude that the random variable $\fq{\xpr,\ypr,\mpr}$ is
independent of $\tq{x,\ell}\, \tq{\xpr,\ell}\, \bbE (\fq{x,y,m}-\fa{x,y,m}\given\calA)$. The same is
thus also true of $\bbE (\fq{\xpr,\ypr,\mpr}-\fa{\xpr,\ypr,\mpr}\given\calA)$, and the claim follows, since the
latter has
mean zero.

\medskip
A second observation is that, for any $A\subseteq \bbZ^{d+1}$ and the
corresponding cylindrical $\sigma$-algebra $\calA = \sigma\setof{V(z)}{z\in A}$,
\be  
\label{eq:observation}
\bbE \bigl[ \tq{x,\ell}\, \tq{\xpr,\ell}\, \bbE (\fq{x,y,m}\given\calA)\, \bbE
(\fq{\xpr,\ypr,\mpr}\given\calA) \bigr] \leq \bbE \bigl[ \tq{x,\ell}\, \tq{\xpr,\ell}\,
\fq{x,y,m}\, \fq{\xpr,\ypr,\mpr} \bigr].
\ee
Indeed, define $g = (\lambda+\log(2d+2))(2\ell+m+\mpr)-h\cdot(y+\ypr)$ and 
let {$\Sigma^*$ be the sum over all the paths 
$\gamma\in\calT(0,x;\ell), \eta\in\calF (x,y;m), \gamma'\in\calT(0,\xpr;\ell)$ and 
$\eta'\in\calF (\xpr,\ypr;\mpr)$. }
Then the attractivity property~\eqref{eq:attractive} implies that
\begin{align*}
&{\rm e}^{g}\, 
\bbE \bigl[ \tq{x,\ell}\, \tq{\xpr,\ell}\, \bbE
(\fq{x,y,m}\given\calA)\, \bbE
(\fq{\xpr,\ypr,\mpr}\given\calA) \bigr]\\
&
\qquad=
{\sum}^*
\prod_{u\in A}
{\rm 
e}^{-\phi_\beta(\ell_{\gamma}(u)+\ell_{\gamma'}(u)+\ell_{\eta}(u)+\ell_{\eta'}
(u))}\,
\prod_{v\not\in A}{\rm 
e}^{-\phi_\beta(\ell_{\eta}(v))-\phi_\beta(\ell_{\eta'}(v)) 
- \phi_\beta ( \ell_{\gamma}(v)+\ell_{\gamma'}(v ))
}\\
&\qquad\leq
{\sum}^*
\prod_{u\in\Zdo}
{\rm 
e}^{-\phi_\beta(\ell_{\gamma}(u)+\ell_{\gamma'}(u)+\ell_{\eta}(u)+\ell_{\eta'}
(u))}
= {\rm e}^{g}\,  
\bbE \bigl[ \tq{x,\ell}\, \tq{\xpr,\ell}\,
\fq{x,y,m}\, \fq{\xpr,\ypr,\mpr} \bigr] .
\end{align*}

\smallskip
Note that~\eqref{eq:observation} implies, in particular, that
\begin{multline}
\label{eq:L2cleaner}
\bigl| \bbE \bigl[ \tq{x,\ell}\, \tq{\xpr,\ell}\, \bbE (\fq{x,y,m}-\fa{x,y,m}\given\calA)\, \bbE
(\fq{\xpr,\ypr,\mpr}-\fa{\xpr,\ypr,\mpr}\given\calA) \bigr] \bigr|\\
\leq 2\, \bbE \bigl[ \tq{x,\ell}\, \tq{\xpr,\ell}\,
\fq{x,y,m}\, \fq{\xpr,\ypr,\mpr} \bigr]\, \One_{\{D(x,y)\cap
D(\xpr,\ypr)\neq\emptyset\}}.
\end{multline}

\subsection{Getting rid of the last irreducible steps}

For several paths $\gamma_1, \dots , \gamma_k$, define
\[
 \Phi_\beta (\gamma_1, \dots ,\gamma_k ) = 
\sum_{u\in\Zdo} \phi_\beta \Bigl( \sum_1^k \ell_{\gamma_{i}} (u ) \Bigr) .
\]
Applying~\eqref{eq:attractive} once more, we see that
\be 
\label{eq:PhiUpper}
\lbr  \Phi_\beta(\gamma,\gamma',\eta,\eta')  - 
\Phi_\beta(\gamma,\gamma') \rbr +\lbr  (m +\mpr )\phi_\beta(1) -  \Phi_\beta
(\eta ) - \Phi_\beta (\eta' )\rbr 
\geq 0 ,
\ee
uniformly in all paths $\gamma\in\calT(0,x;\ell)$,
$\gamma'\in\calT(0,\xpr;\ell)$, $\eta\in\calF (x,y;m)$
and $\eta'\in\calF (\xpr,\ypr;\mpr)$. This implies that
\[
\sum_{\substack{\eta\in\calF (x,y;m),\\\eta'\in\calF (\xpr,\ypr;\mpr)}} {\rm 
e}^{-\Phi_\beta(\gamma,\gamma',\eta,\eta')} \leq
{\rm e}^{\phi_\beta(1)(m+\mpr)}\,{\rm e}^{-\Phi_\beta(\gamma,\gamma')}
\sum_{\substack{\eta\in\calF (x,y;m),\\\eta'\in\calF (\xpr,\ypr;\mpr)}} {\rm 
e}^{-\Phi_\beta(\eta)-\Phi_\beta(\eta')}.
\]
Since $\lim_{\beta\to 0} \phi_\beta(1) = -\log(1-p_\infty)$, it follows
from~\eqref{eq:L2cleaner} and~\eqref{eq:exp-tails} that, in the very weak
disorder regime,
\begin{align*}
\bigl| \bbE \bigl[ \tq{x,\ell}\, \tq{\xpr,\ell}\,
\bbE (\fq{x,y,m}-\fa{x,y,m}&\given\calA)\,
\bbE(\fq{\xpr,\ypr,\mpr}-\fa{\xpr,\ypr,\mpr}\given\calA) \bigr]\bigr|\\
&\leq
2 {\rm e}^{\phi_\beta(1)(m+\mpr)}\,
\bbE \bigl[ \tq{x,\ell}\, \tq{\xpr,\ell} \bigr]\,
\One_{\{D(x,y)\cap D(\xpr,\ypr)\neq\emptyset\}}\,
\fa{x,y,m}\fa{\xpr,\ypr,\mpr}\\
&\leq
2{\rm e}^{-(\nu-\phi_\beta(1))(m+\mpr)}\,
\One_{\{D(x,y)\cap D(\xpr,\ypr)\neq\emptyset\}}\,
\bbE\bigl[ \tq{x,\ell}\, \tq{\xpr,\ell} \bigr]\\
&\leq
2{\rm e}^{-(\nu/4)(m+\mpr) - (\nu/4)(\abs{x-\xpr})}\,
\bbE\bigl[ \tq{x,\ell}\, \tq{\xpr,\ell} \bigr].
\end{align*}
Indeed, in order for the event $D(x,y)\cap D(\xpr,\ypr)\neq\emptyset$ to occur, it is necessary that $\abs{x-\xpr}\leq
m\vee \mpr \leq m+\mpr$.

\subsection{Weakly interacting random walks}
There remains to prove that
\begin{equation}
\label{eq:L2TargetEstimate}
\bbE \bigl[\tq{x,\ell}\tq{\xpr,\ell}\bigr]
\leq
\frac{c_1}{\ell^{d+1 -\rho }} \exp\Bigl\{ - c_2\frac{\abs{x -\ell v}^2}{\ell}
- c_2\frac{\abs{\xpr -\ell v}^2}{\ell}\Bigr\}.
\end{equation}
Let us denote by $\RWP$ the law of the random walk on $\Zdo\times\bbN$, whose
increments have law $(\fa{x,n})_{x\in\Zdo,n\in\bbN}$. We shall denote its
(random) trajectory by the couple $(\RWX,\RWL)$, with
$\RWX=(X_0=0,X_1,X_2,\ldots)$ and $\RWL=(L_0=0,L_1,L_2,\ldots)$. With these
notations, we can write
\[
\bbE \bigl[ \tq{x,\ell} \bigr]
= \ta{x,\ell}
= \RWP(\exists k:\, (X_k,L_k)=(x,\ell)).
\]
In general, the left-hand side of~\eqref{eq:L2TargetEstimate} does
not allow for a similar expression. Notice, however, that the attractivity
property implies the lower bound
\[
\bbE \bigl[ \tq{x,\ell}\, \tq{\xpr,\ell} \bigr] \geq
\ta{x,\ell}\ta{\xpr,\ell}\\
=
\RWPC\bigl(\exists k,\kpr:\, (X_k,L_k)=(x,\ell), (\Xpr_{\kpr},
\Lpr_{\kpr})=(\xpr,\ell)\bigr) ,
\]
where $\RWPC$ denotes the law of a couple of independent random walks
$(\RWX,\RWL)$ and $(\RWXpr,\RWLpr)$ as above.
It is important to observe that~\eqref{eq:L2TargetEstimate} would be an immediate consequence of the local limit
theorem for random walks if its left-hand side was replaced by $\ta{x,\ell}\ta{\xpr,\ell}$. To
prove~\eqref{eq:L2TargetEstimate}, we thus have to prove that, in the very weak disorder regime, this local limit
behaviour is not destroyed by the effective attractive interaction between the two paths resulting from averaging
$\tq{x,\ell}\tq{\xpr,\ell}$ over the disorder.

To facilitate the notation, define the events $R_{x,\ell} =\lbr \exists k :\,
(X_k,L_k)=(x,\ell)\rbr$ and $R_{\xpr,\ell} =\lbr \kpr:\, (\Xpr_k,L_k)=(\xpr
,\ell)\rbr$. Then, by the very same attractivity property of the potential, 
\begin{equation}
\label{eq:AttrBound} 
\bbE \bigl[ \tq{x,\ell}\, \tq{\xpr,\ell} \bigr] \leq 
\RWEC \bigl[ {\rm e}^{\Delta_\beta (\RWX ,\RWXpr
)}\One_{R_{x,\ell}}\One_{R_{\xpr,\ell}} \bigr], 
\end{equation}
where
\begin{equation}
\label{eq:Delta-beta}
 \Delta_\beta (\RWX ,\RWXpr ) = 
\log\Bigl\{ \sumtwo{\gamma\sim\RWX}{\gamma^\prime\sim\RWXpr} {\rm
e}^{\Phi_\beta (\gamma )+\Phi_\beta (\gamma^\prime ) - \Phi_\beta (\gamma ,
\gamma^\prime )}\Bigr\} . 
\end{equation}
Fix $\rho < 1/12$. Formula~(4.13 ) in~\cite{IV-Crossing} implies that
\[
\RWEC {\rm e}^{\frac{d+1}{\rho} \Delta_\beta (\RWX ,\RWXpr )} \leqs 1
\]
in the very weak disorder regime. Therefore, by H\"{o}lder inequality, 
\[
\bbE \bigl[ \tq{x,\ell}\, \tq{\xpr,\ell} \bigr] \leqs 
\lb \ta{x,\ell}\ta{\xpr,\ell}\rb^{\frac{d+1 -\rho}{d+1}} .
\]
The target~\eqref{eq:L2TargetEstimate} follows now by~\eqref{eq:Aupper}. \qed

\ACKNO{The authors would like to thank an anonymous referee for a very careful reading of 
the paper and for constructive criticism and suggestions on how the exposition should be 
clarified and improved.\newline
YV was partially supported by the Swiss National Science Foundation. DI was supported
by the Israeli Science Foundation grant 817/09. }
\end{document}